\documentclass[11pt,letterpaper,leqno]{amsart}
\usepackage[pdftex]{color,graphicx} 
\usepackage{graphics,amsmath,amsfonts,amsthm,amssymb}
\usepackage[margin=30pt,font=small,labelfont=bf]{caption}

\title{On the Classification of Legendrian Rational Tangles}
\author{Gregory R. Schneider}

\def\rr{{\mathbb{R}}}
\def\zz{{\mathbb{Z}}}

\def\D{{\mathbb{D}}}
\def\B{{\mathbb{B}}}
\def\d{{\mathcal{D}}}
\def\b{{\mathcal{B}}}
\def\R{{\mathcal{R}}}
\def\r{{R}}
\def\G{{\Gamma}}
\def\g{{\gamma}}
\def\plK{{\overline{K}}}
\def\plB{{\overline{B}}}
\def\wbD{{\widetilde{D}}}
\def\pD{{\widehat{D}}}
\def\PG{{\widetilde{\Gamma}}}

\def\wbE{{\widetilde{E}}}
\def\pE{{\widehat{E}}}

\numberwithin{equation}{section}
\numberwithin{figure}{section}

\theoremstyle{plain} \newtheorem{mainthm}{Theorem}[section]
\theoremstyle{plain} \newtheorem{regcontfrac}[mainthm]{Proposition}
\theoremstyle{plain} \newtheorem{bduniqueness}[mainthm]{Proposition}
\theoremstyle{plain} \newtheorem{bdcount}[mainthm]{Proposition}
\theoremstyle{plain} \newtheorem{bddtangle}[mainthm]{Proposition}
\theoremstyle{plain} \newtheorem{unknotKq}[mainthm]{Proposition}
\theoremstyle{plain} \newtheorem{squareunknots}[mainthm]{Lemma}
\theoremstyle{plain} \newtheorem{Kqtbandr}[mainthm]{Proposition}
\theoremstyle{plain} \newtheorem{cardinalitylemma}[mainthm]{Lemma}
\theoremstyle{plain} \newtheorem{bijectionlemma}[mainthm]{Lemma}
\theoremstyle{plain} \newtheorem{trivialflypes}[mainthm]{Proposition}
\theoremstyle{plain} \newtheorem{verticalflypes}[mainthm]{Theorem}
\theoremstyle{plain} \newtheorem{oddlengthhorizontalflypes}[mainthm]{Theorem}
\theoremstyle{plain} \newtheorem{infconnhorizontalflypes}[mainthm]{Theorem}

\begin{document}
\begin{abstract}
We show that under certain conditions the flyping operation on rational tangles, which produces topologically isotopic tangles, may also produce tangles which are not Legendrian isotopic when viewed in the standard contact structure on $\rr^3$.  This work is motivated by questions posed by Traynor, and incorporates techniques inspired by work of Eliashberg and Fraser.  The proofs within employ a new method of diagramming rational tangles which can also be used to easily describe the characteristic foliations of compressing discs for their complements.
\end{abstract}
\maketitle

\thispagestyle{empty}
\section{Introduction}
\label{sec:introduction}

In~\cite{T98} and~\cite{T01}, Traynor examines a class of two-strand tangles in $\rr^3$ whose $S^1$-closures produce two-component links in the 1-jet space of the circle $\mathcal{J}^1(S^1)$.  Subject to the restriction that they be everywhere tangent to a particular 2-plane distribution on $\rr^3$, which induces a similar distribution on $\mathcal{J}^1(S^1)$, the question arises whether such tangles are isotopic under a classical topological operation known as a flype.  Approaching this question via generating functions, Traynor produces in~\cite{T01} the foundation of a classification of these tangles.  However, she also poses a number of open questions regarding the extension of this classification beyond her results.  It is this set of questions that provides the initial motivation for this work.

In~\cite{EF08}, Eliashberg and Fraser establish a complete classification of unknots in $\rr^3$ which obey the same tangency condition relative to the aforementioned distribution.  Their approach uses manipulations of foliations induced by this distribution on spanning discs to construct isotopies between unknots which bear the same classical invariants.  The techniques we employ later in this work are inspired by this study of disc foliations.  By examining discs which separate the strands of related tangles in $\rr^3$, with the additional restriction that their bounding unknots remain fixed through any deformation, we show that the corresponding foliations of these discs are capable of capturing information about the tangles themselves.  We then show that it is possible to use these foliations as a distinguishing invariant of Legendrian rational tangles, producing a different sort of classification than that developed by Traynor.  

More precisely, we consider a class of Legendrian rational tangles in the standard contact structure $\xi_{std}$ on $\rr^3$.  These tangles consist of pairs of properly embedded arcs in a closed 3-dimensional ball, constructed topologically via a sequence of alternating horizontal and vertical twists, subject to both the tangency condition described above and a regularity condition that we will define later.  The discs we consider, disjoint from the corresponding tangles, are also properly embedded in this ball, bounded by simple closed curves on its boundary which separate the endpoints of the corresponding tangles into pairs in either component of the complement of this curve.  We study all of these objects up to ambient isotopies of the 3-ball which fix its boundary and preserve both the contact structure on this ball and the Legendrian nature of the tangles throughout.  In this setting, the topological flype manifests as two kinds of Legendrian flypes when viewed as modifications of planar projections of tangles.   The original goal of this work was to produce a complete classification of Legendrian rational tangles under Legendrian flypes.  However, the results contained in what follows fall short of this goal.  Instead we obtain an intermediate result, more comprehensive in the class of tangles studied and free from a parity concern on the number of flypes that restricted Traynor's classification in ~\cite{T01}, yet still limited by certain artifacts of our techniques.  The extent of this classification is summarized in Theorem~\ref{mainthm}.

\begin{mainthm}
\label{mainthm}
For any positive rational number $q$, let $\G_{q^f}$ and $\G_{q^g}$ be regular Legendrian rational tangles obtained from a standard tangle $\G_q$ via some sequence of flypes, where $\G_q$ consists of $n$ twist components.
\begin{itemize}
\item For any $n$, if $\G_{q^f}$ and $\G_{q^g}$ differ only by vertical Legendrian flypes, then they are Legendrian isotopic.
\item For any $n$, if $\G_{q^f}$ and $\G_{q^g}$ are obtained from $\G_q$ by distinct numbers of horizontal Legendrian flypes performed at self-crossings of one of the strands, then they are not Legendrian isotopic.
\item For $n$ odd, if $\G_{q^f}$ and $\G_{q^g}$ are obtained from $\G_q$ by distinct numbers of horizontal Legendrian flypes, then they are not Legendrian isotopic. 
\end{itemize}
\end{mainthm}

The remainder of our exposition is organized as follows.  In Section~\ref{sec:rationaltangles} we provide definitions of the topological constructs we will require, including a variation on Conway's construction of rational tangles from~\cite{C70}.  In Section~\ref{sec:contactgeometry} we introduce relevant features of contact geometry, taken largely from Etnyre's surveys~\cite{E03} and~\cite{E05}, closing with a discussion akin to Traynor's in~\cite{T98} coupling our conventions on rational tangles with the contact geometric setting.  In Section~\ref{sec:boxdotdiagrams} we describe the construction of \emph{box-dot diagrams}---a notational tool that will allow us to easily describe the battery of constructions involved in our work.  Section~\ref{sec:tanglesandfoliations} contains a pair of technical lemmas which enable us to utilize box-dot diagrams efficiently to study the effects of flyping on a Legendrian tangle.  In Section~\ref{sec:results} we provide precise statements and proofs of the classification results that comprise Theorem~\ref{mainthm}.  Finally, in Section~\ref{sec:conclusion} we discuss the limitations of our approach and possible directions for future research.

\section{Rational Tangles}
\label{sec:rationaltangles}
For a rational number $q$, a topological rational tangle $G_q$ consists of a pair of properly embedded, disjoint, locally unknotted arcs---the \emph{strands} of the tangle---in a closed ball $B$ in $\rr^3$.  We say two tangles are \emph{tangle isotopic} if there is an ambient isotopy of $B$, fixed along its boundary $\partial B$, carrying one tangle to the other.  

\subsection{Construction of Rational Tangles}
\label{sec:tangleconstruction}
Conway~\cite{C70} originally described the construction of a rational tangle $G_q$ from a vector representation of $q$ obtained as in Formula~\eqref{continuedfraction}.  The notation we use for this vector comes from~\cite{T98}. 
\begin{equation}
\label{continuedfraction}
(q_n,q_{n-1},\ldots,q_1) \sim q = q_1+\cfrac{1}{q_2+\cfrac{1}{\ddots+\cfrac{1}{q_n}}}
\end{equation}

We now include a brief review of this construction, employing a modified convention taken from~\cite{KL02}.  For convenience, we will view the page as the $xz$-plane, and represent tangles in $\rr^3$ via their projections into this plane.  We begin by fixing four distinct points on $\partial B$.  These points will serve as the endpoints of the strands.  The construction of any  rational tangle initiates from one of two specific tangles, the \emph{0-tangle} $G_0$ or the \emph{$\infty$-tangle} $G_{\infty}$, shown in Figure~\ref{0inftangles}.  It will later be useful to distinguish between the two strands of a tangle visually by way of a coloring of the individual strands.  Any such coloring is induced by first choosing a distinct color for each of the strands in either $G_0$ or $G_{\infty}$ at the outset of the following construction.

\begin{figure}[ht]
\begin{center}
\includegraphics[width=0.5\textwidth]{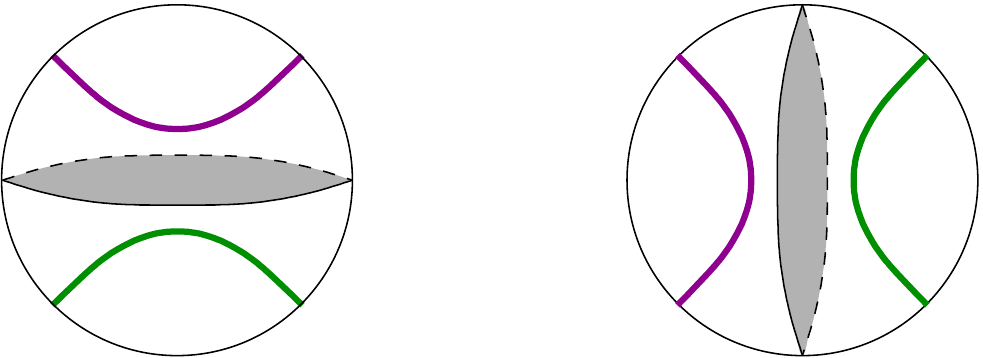}
\caption[The tangles $G_0$ and $G_{\infty}$.]{The initial tangles $G_0$ (left) and $G_\infty$ (right).}
\label{0inftangles}
\end{center}
\end{figure}

Given a vector $(q_n,q_{n-1},\ldots,q_1)\sim q$, assume first that $n$ is odd.  Accordingly, we will refer to the resulting tangle $G_q$ as an \emph{odd-length tangle}.  In this case we begin with the 0-tangle $G_0$.  Corresponding to the first component $q_n$, we employ an isotopy of the ball $B$ which fixes the left endpoints, while applying $q_n$ half-twists to the right endpoints.  We adopt the convention that this \emph{horizontal twisting} shall be performed in a counterclockwise fashion with respect to the positive $x$-axis for positive components; clockwise for negative components.  We next employ an isotopy of $B$ which fixes the upper endpoints, while applying $q_{n-1}$ half-twists to the lower endpoints.  Similarly, here we require that this \emph{vertical twisting} shall be performed in a counterclockwise fashion with respect to the positive $z$-axis for positive components; clockwise for negative components.  Figure~\ref{twistconvention} illustrates the results of performing three horizontal twists or two vertical twists under these conventions.  To complete the construction of the tangle $G_q$ we proceed iteratively through the remaining components of the vector $(q_n,q_{n-1},\ldots,q_1)$, alternating between twisting either the right endpoints or the bottom endpoints according to components $q_j$ with $j$ odd or even, respectively.  

\begin{figure}[ht]
\begin{center}
\includegraphics[width=0.4\textwidth]{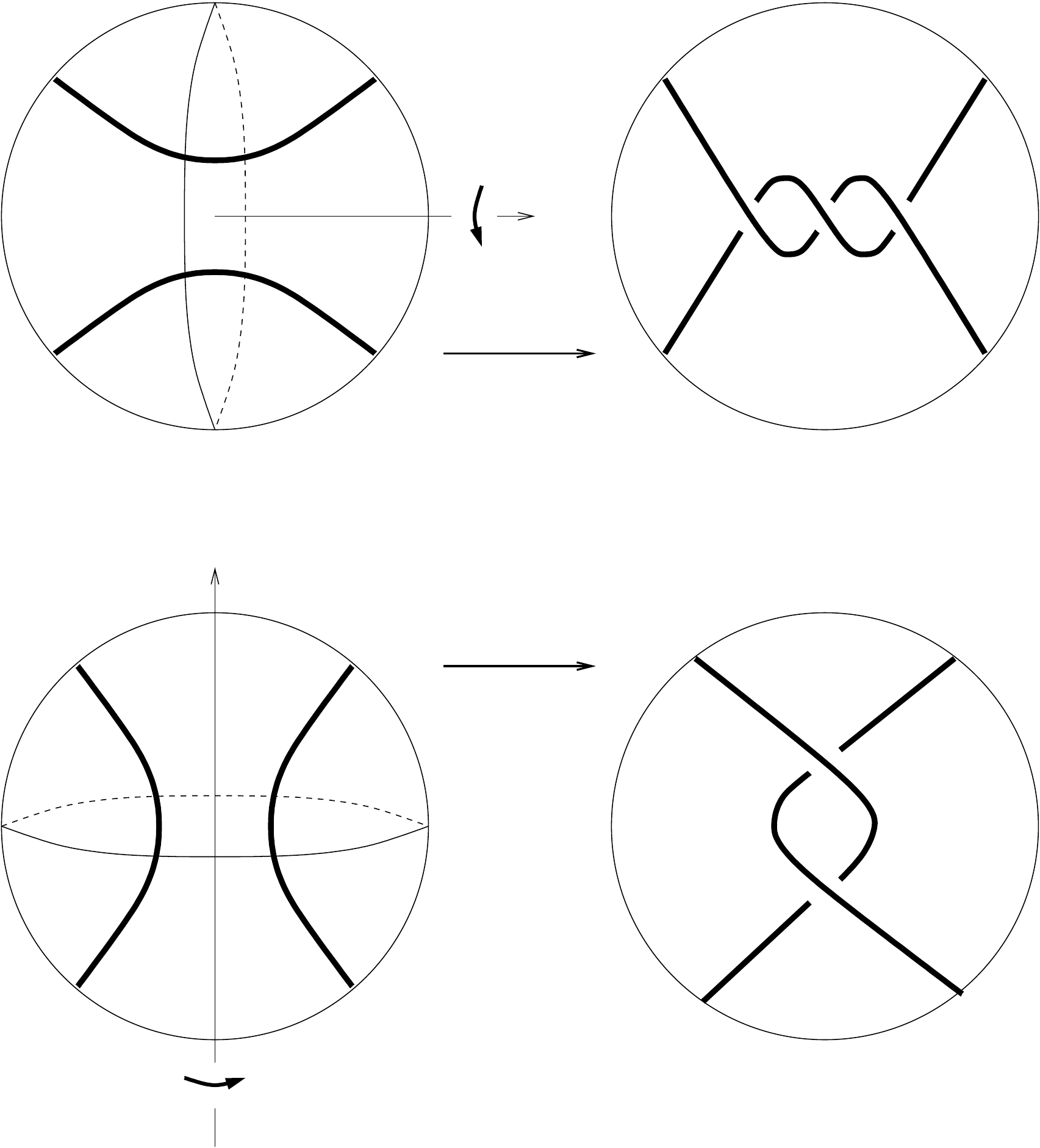}
\caption[The twisting convention for $G_q$.]{Positive horizontal (top) and vertical (bottom) twists in constructing $G_q$ produce crossings in which the overstrand has a more negative slope when viewed in the $xz$-plane.}
\label{twistconvention}
\end{center}
\end{figure}

To construct an \emph{even-length tangle}, we instead begin with the $\infty$-tangle $G_{\infty}$ and proceed iteratively through the components of the given vector, alternating between vertical twisting according to even-indexed components, starting with $q_n$, and horizontal twisting according to odd-indexed components.  Under these conventions, the component $q_1$ will always correspond to horizontal twisting in the final step of this construction, regardless of the parity of $n$.  Figure~\ref{ratltangles} provides examples of rational tangles arising from this construction.

\begin{figure}[ht]
\begin{center}
\includegraphics[width=0.5\textwidth]{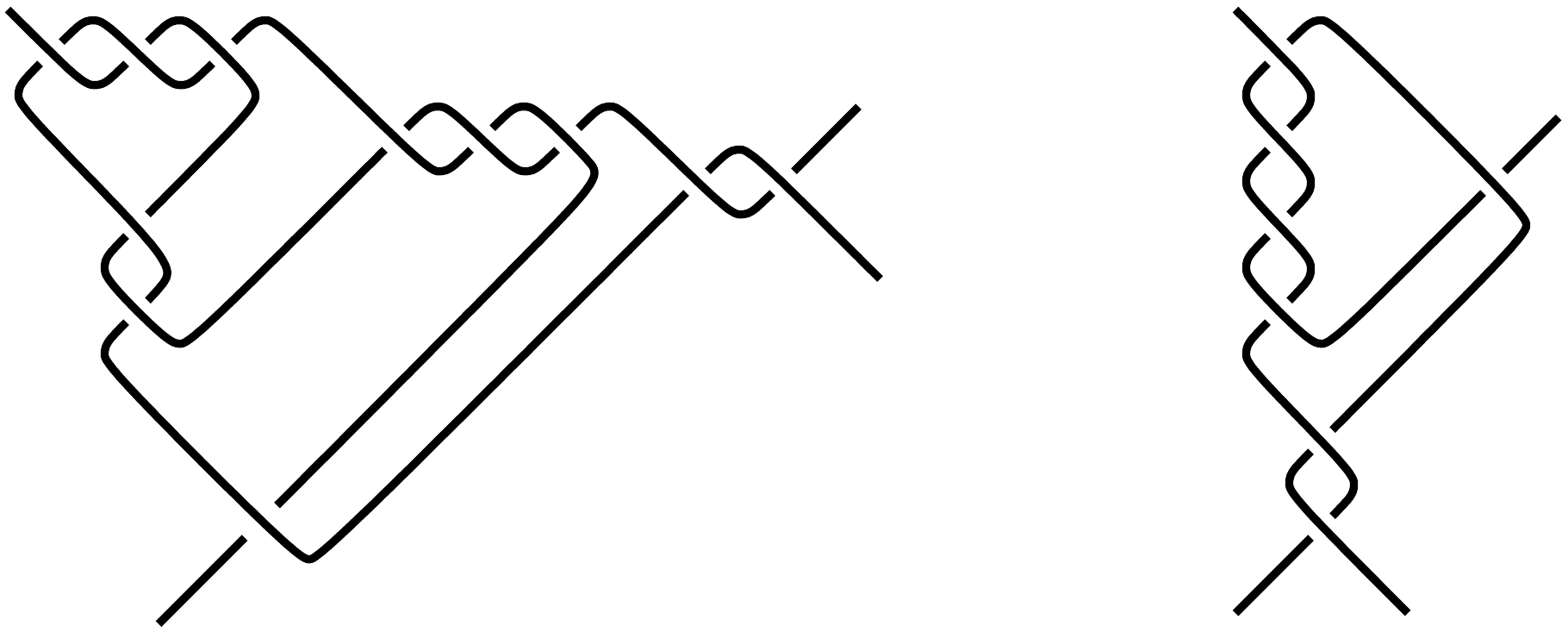}
\caption[The rational tangles $G_{(3,2,3,1,2)}$ and $G_{(4,1,2,0)}$.]{Examples of (topological) rational tangles.  Shown are the tangles $G_{(3,2,3,1,2)}$ (left) and $G_{(4,1,2,0)}$ (right).}
\label{ratltangles}
\end{center}
\end{figure}

A \emph{rational tangle} is then any tangle which is tangle isotopic to a tangle constructed as above.  We remark here that our conventions regarding how we define positive and negative twisting are reversed from~\cite{C70}.  Nevertheless, Kaufmann and Lambropoulo~\cite{KL02} prove that this modification produces an equivalent theory, one which still bears Conway's classical result that two such tangles are isotopic if and only if the continued fractions corresponding to their vectors produce the same rational number.

It is a classical result of number theory that every positive rational number $q$ has a unique continued fraction representation corresponding to a vector $(q_n,q_{n-1},\ldots,q_1)$ whose components satisfy $q_n \geq 2$, $q_1 \geq 0$, and $q_j \geq 1$ for $2\leq j < n$.  Hardy and Wright~\cite{HW71} refer to such a representation as a \emph{regular continued fraction} for $q$.  For the remainder of this work, we will consider only tangles which are constructed from the regular continued fraction of a positive rational number.

\subsection{Compressing Discs for Rational Tangles}
\label{sec:compressingdiscs}
The construction of a rational tangle described above can be extended to produce a related object that will be fundamental in our approach to proving Theorem~\ref{mainthm}.  For a vector $(q_n,q_{n-1},\ldots,q_1) \sim q$, we follow the construction of $G_q$, paying particular attention to the discs shown in Figure~\ref{0inftangles} which separate the strands of either initial tangle $G_0$ or $G_{\infty}$.  Each isotopy of the ball $B$ arising from the twisting of the endpoints of $G_q$ then induces a corresponding isotopy of this disc, as in Figure~\ref{compression}.  We may assume this disc remains embedded in $B$ and disjoint from the strands of $G_q$ throughout this process.  Once the construction of $G_q$ is complete, we are left with a properly embedded disc $D \subset B-G_q$ whose boundary does not bound a disc in $\partial B - \partial G_q$.  Such a disc is known as a \emph{compressing disc} for $B-G_q$. 

\begin{figure}[ht]
\begin{center}
\includegraphics[width=\textwidth]{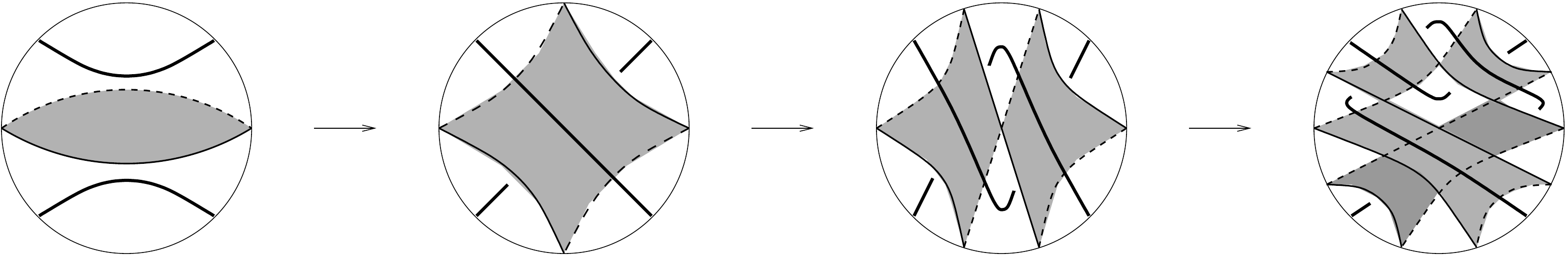}
\caption[Constructing compressing discs for rational tangles.]{A compressing disc for $B-G_{(2,1,0)}$.}
\label{compression}
\end{center}
\end{figure}  

\subsection{Flypes of Rational Tangles}
\label{sec:topflypes}
The topological concept of a \emph{flype} appears in Conway's work~\cite{C70} on rational tangles, though this term is originally due to Tait, dating back to the later half of the $19^{\textrm{th}}$ century.  A flype is performed by first enclosing a portion of a topological tangle in a sphere in such a way that the sphere is punctured exactly four times by the strands of the tangle.  Referring to a planar projection of the tangle, we also require that two adjacent arcs outside of the sphere immediately cross.  The flype is then performed by rotating the sphere in such a manner that the crossing is undone, resulting in the formation of a new crossing between the other two strands, as shown in Figure~\ref{topflypes}.  Clearly, this operation corresponds to an ambient topological isotopy of the tangle.  While we illustrate only positive twists in Figure~\ref{topflypes}, we can extend this to accommodate tangles with negative twist components as well by reversing the crossings shown in this figure.  In preparation for Section~\ref{legflypes}, we identify two species of flype in the figure, despite there being no real topological distinction between them.

\begin{figure}[ht]
\begin{center}
\resizebox{0.65\textwidth}{!}{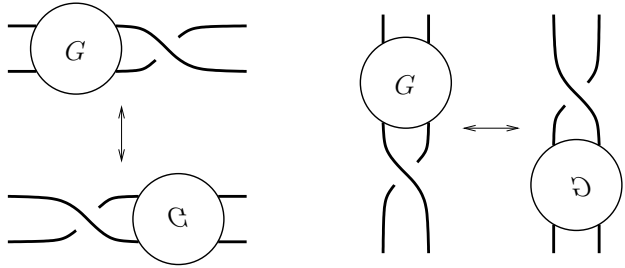}
\caption[Topological flypes.]{``Horizontal'' (left) and ``vertical'' (right) flypes in the topological setting.}
\label{topflypes}
\end{center}
\end{figure}  

To specify a tangle which has been subjected to this flyping operation, we borrow a notational convention from~\cite{T98}.  By construction, every crossing in a planar diagram of a rational tangle $G_q$ constructed as above corresponds to an available flype, where the aforementioned sphere contains the \emph{subtangle} preceding the crossing used in the flype.  Thus for each twist component $q_j$ we may perform as many as $q_j$ distinct flypes, or none at all, at the corresponding set of crossings in the diagram.  The vector $(q_n^{f_n},q_{n-1}^{f_{n-1}},\ldots,q_1^{f_1})$ will be used to denote the tangle obtained from $G_q$ by performing a total of $f_j$ flypes at each twist component $q_j$, with $0\leq f_j\leq q_j$ for $1\leq j\leq n$.  In what follows we will frequently abbreviate $(q_n^{f_n},q_{n-1}^{f_{n-1}},\ldots,q_1^{f_1})$ simply as $q^f$.

\section{Contact Geometry}
\label{sec:contactgeometry}
Our ambient setting for any contact geometric concerns will be $(\rr^3,\xi_{std})$, consisting of $\rr^3$ with its \emph{standard contact structure} $\xi_{std}=dz-y\, dx$.  This contact structure can be visualized as in Figure~\ref{r3} by observing that the condition $dz-y\, dx=0$ implies that the contact plane at any point $(x,y,z)$ is spanned by $\{\partial/\partial y, \partial/\partial x + y\, \partial/\partial z\}\subset T_{(x,y,z)}\rr^3$.  The first vector in this set is always parallel to the $y$-axis, while the second is a vector of slope $y$ when projected to the $xz$-plane.  Thus along the $y$-axis the contact planes are flat at the origin and rotate about this axis in a left-handed manner, approaching vertical only as $y \rightarrow \pm\infty$.  The rest of the contact structure can then be realized by translating the planes along this axis in the $x$- and $z$-directions.  We will assume that this contact structure is endowed with a smooth upward co-orientation.

\begin{figure}[ht]
\begin{center}
\resizebox{0.4\textwidth}{!}{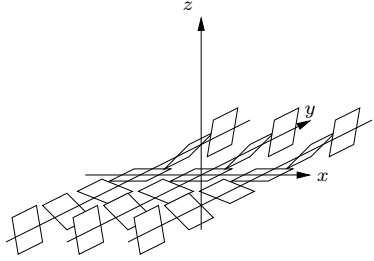}
\caption[The standard contact structure $(\rr^3,\xi_{std})$.]{A portion of $\xi_{std}$ on $\r^3$.}
\label{r3}
\end{center}
\end{figure}

Throughout this work we will make use of a number of fundamental tools from the theory of contact geometry, which we will now proceed to describe.  Much of the information contained in the following subsections can be found in either~\cite{E03} or~\cite{E05}, though some of this has been adapted to our particular setting. 

\subsection{Front Projections, Legendrianization, and Lifting}
\label{sec:legendrianization}
A curve in $(\rr^3,\xi_{std})$ is said to be \emph{Legendrian} if it is everywhere tangent to the contact planes.  Given a Legendrian curve in $\rr^3$, its \emph{front projection} is the image of the curve under the map from $\rr^3$ to $\rr^2$ given by $(x,y,z)\mapsto (x,z)$.  Double-points in the front projection of a generic Legendrian curve will be transverse.  By the nature of the contact planes, the Legendrian condition forces us to resolve each of these double-points to a crossing whose overstrand has a more negative slope than its understrand.  Further, since the planes of the standard contact structure are never vertical, front projections can contain no vertical tangencies.  In place of any vertical tangencies that may occur in a topological projection, front projections will instead contain semi-cubical\footnote{I.e., cusps which resemble the graph of the equation $y^2=x^3$ near the origin.} \emph{cusps}.  Each cusp point corresponds to a point along the Legendrian curve where the tangent to the curve is parallel to the y-axis.

Conversely, if a planar diagram of a curve satisfies these two conditions---that double-points of the diagram are isolated, transverse, and resolve to crossings in which the overstrand has a more negative slope than the understrand, and the diagram contains cusp points in place of any vertical tangencies---then the diagram represents the front projection of some Legendrian curve.  Given any planar diagram of a topological curve, we can \emph{Legendrianize} the projection by rotating any crossings to match the first of these conditions, if necessary, and by replacing any vertical tangencies with cusps to obtain a valid front projection, as shown in~\cite{E05}.  If the Legendrianized diagram is parametrized with coordinate functions $(x(t),z(t))$, then we can recover the corresponding Legendrian curve via the \emph{Legendrian lifting map} $(x(t),z(t))\mapsto(x(t), dz/dx|_t, z(t))$.  This lifting map serves as a fundamental tool in studying Legendrian curves, as it allows us to represent these curves precisely via their front projections.

By way of the Legendrian lifting map, we may realize any Legendrian isotopy of a Legendrian curve in $(\rr^3,\xi_{std})$ by way of a transverse\footnote{We borrow this terminology from Eliashberg~\cite{E87}.  ``Transverse'' here refers to the nature of double-points in the projection throughout the isotopy, and should not be confused with the notion of \emph{transverse curves}---curves transverse to the contact planes---in $(\rr^3,\xi_{std})$.} isotopy of the front projection which preserves the two conditions mentioned above.  

\subsection{Legendrian Flypes of Legendrian Rational Tangles}
\label{sec:legflypes}
By our conventions, the planar projection of a rational tangle $G_q$ with positive twist components already satisfies the first requirement of a front projection.  In Legendrianizing the planar diagram of $G_q$, we will adhere to the additional convention that we replace any vertical tangencies with cusps whose tangent line at the cusp point is horizontal, as in Figure~\ref{legratltangles}.  The resulting front projection can be lifted to a \emph{Legendrian rational tangle} $\G_q$ via the Legendrian lifting map.

\begin{figure}[ht]
\begin{center}
\includegraphics[width=0.5\textwidth]{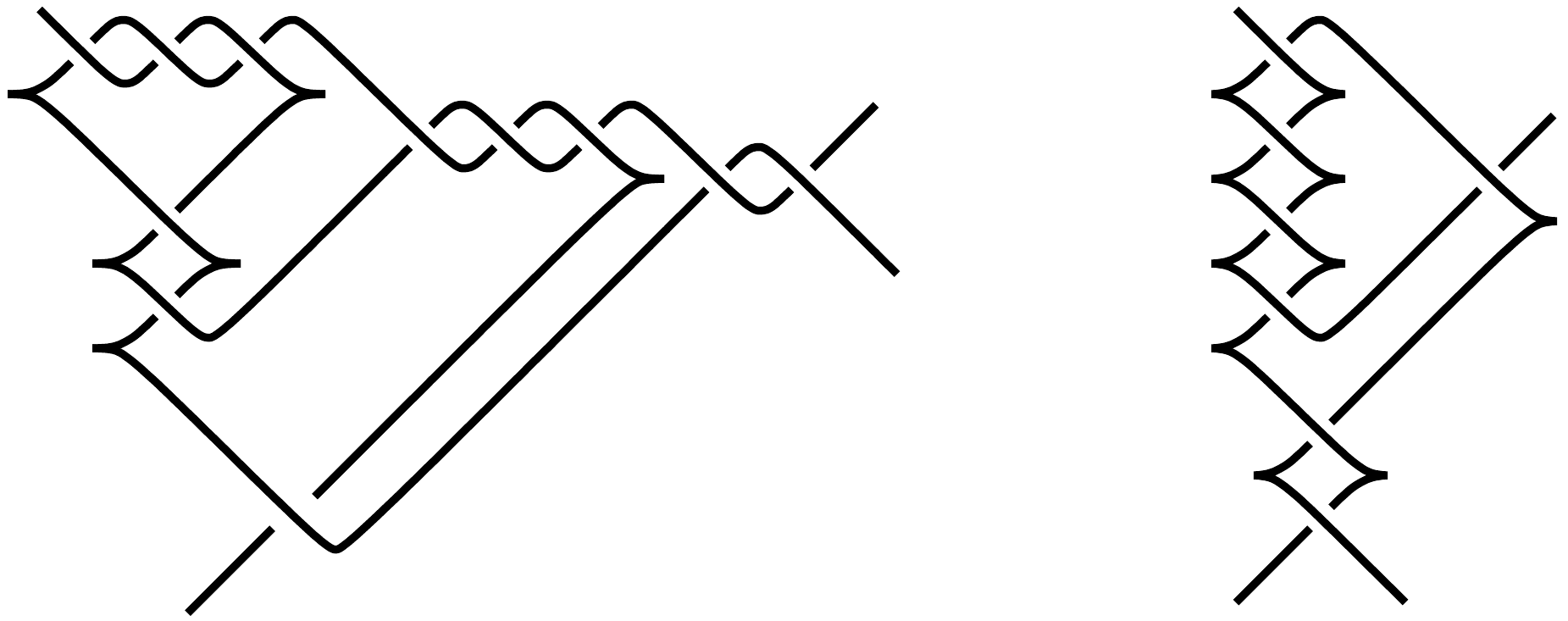}
\caption[The Legendrian rational tangles $\G_{(3,2,3,1,2)}$ and $\G_{(4,1,2,0)}$.]{Examples of Legendrian rational tangles (cf.\ Figure~\ref{ratltangles}).  Shown are the tangles $\G_{(3,2,3,1,2)}$ (left) and $\G_{(4,1,2,0)}$ (right).}
\label{legratltangles}
\end{center}
\end{figure}

In the same manner, we can lift Legendrianized versions of the flype operations shown in Figure~\ref{topflypes} to the \emph{Legendrian flypes} shown in Figure~\ref{legflypes}.  In contrast to their topological counterparts, there is a more visible distinction between \emph{vertical} and \emph{horizontal} Legendrian flypes.  The primary purpose of this paper is to investigate when these two modifications of a front projection can be realized as a Legendrian isotopy of tangles.

\begin{figure}[ht]
\begin{center}
\resizebox{0.65\textwidth}{!}{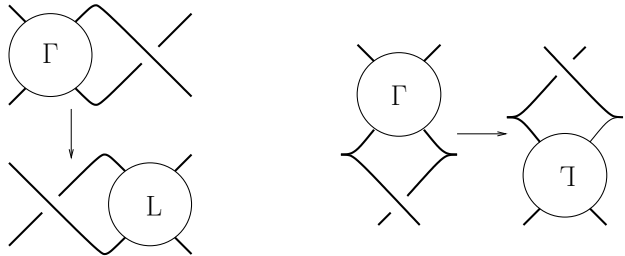}
\caption[Legendrian flypes.]{Horizontal (left) and vertical (right) Legendrian flypes.}
\label{legflypes}
\end{center}
\end{figure}

\subsection{The Classical Invariants $tb$ and $r$}
\label{sec:classicalinvariants}
There are two classical numerical invariants that can be associated to an oriented, nullhomologous Legendrian knot $K$.  The \emph{Thurston-Bennequin number} $tb(K)\in\zz$ measures the net twisting of the contact planes along $K$ relative to a framing on $K$ induced by a Seifert surface.  The \emph{rotation number} $r(K)\in\zz$ measures the winding number of $K$ relative to a trivialization of the tangent bundle $TK$ induced by this Seifert surface.  We mention these descriptions here only to provide some intuitive sense of the geometric significance of these invariants.    

Given an oriented front projection of $K$, we will denote by $wr(K)$ the \emph{writhe}---the number of right-handed crossings minus the number of left-handed crossings---of the projection.  We will also denote by $D$ and $U$ the numbers of downward- and upward-oriented cusps, respectively.  Then $tb(K)$ and $r(K)$ can be computed using the following formulas, as in~\cite{E05}:
\begin{align}
tb(K) &= wr(K)-\textstyle{\frac{1}{2}}(D+U)\label{tbformula}\\
r(K) &= \textstyle{\frac{1}{2}}(D-U)\label{rformula}
\end{align}

It is noted in~\cite{T98} that these classical invariants can be defined similarly for Legendrian arcs with fixed endpoints, providing us with \emph{strandwise classical invariants} for the individual strands of a Legendrian rational tangle.  These strandwise classical invariants can also be computed as in Formulas~\eqref{tbformula} and~\eqref{rformula}, and thus take on either integral or half-integral values.  

\subsection{Characteristic Foliations}
\label{sec:characteristicfoliations}
Given a smooth surface $\Sigma \subset (\r^3,\xi)$, at any point $p\in\Sigma$ we consider the intersection $T_p\Sigma\cap\xi_p$ of the tangent plane to $\Sigma$ at $p$ with the contact plane at $p$.  This intersection will either be transverse, consisting of a line $\ell_p$ through $p$, or the two planes will be identical.  The resulting correspondence $p\mapsto\ell_p$ then defines a singular line field $L$ on $\Sigma$, with singularities at precisely those points where the tangent and contact planes agree.  The \emph{characteristic foliation} $\Sigma_{\xi}$ of a smooth surface $\Sigma$ is the singular foliation of $\Sigma$ consisting of the integral curves of the singular line field $L$.  

\begin{figure}[ht]
\begin{center}
\includegraphics[width=0.25\textwidth]{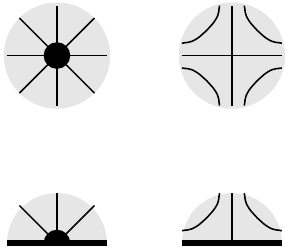}
\caption[Elliptic and hyperbolic singularities.]{Interior (top) and boundary (bottom) variations of elliptic (left) and hyperbolic (right) singularities.}
\label{singularities}
\end{center}
\end{figure} 

Generic singularities of the characteristic foliation are either \emph{elliptic} or \emph{hyperbolic}, illustrated in Figure~\ref{singularities}, and can occur either in the interior of $\Sigma$ or along its boundary.  If the surface $\Sigma$ is oriented, then the singularities are said to be \emph{positive} if the co-orientations of the contact planes agree with those of the tangent planes of $\Sigma$, and \emph{negative} if these co-orientations disagree.  The leaves of the foliation are, \textit{a priori}, Legendrian curves in $(\r^3,\xi_{std})$, and so we may use their front projections to describe a particular embedding of $\Sigma$ into $(\rr^3,\xi_{std})$ by way of the Legendrian lifting map.  We will later exploit this observation to construct particular embeddings of compressing discs, once their characteristic foliations are known.

\section{Box-Dot Diagrams}
\label{sec:boxdotdiagrams}
The diagrams described below were originally devised as a notational tool for organizing the information collected in the early stages of this research.  However, as this work developed it quickly became apparent that they in fact encoded, almost accidentally, much more than they were ever intended to.  While simple to construct, these diagrams contain a surprising amount of information about Legendrian rational tangles and their surrounding space, as we will see in Section~\ref{sec:boxdotconstructions}.  Here we will focus on the construction of box-dot diagrams, and some of their basic properties.

\subsection{Construction of Box-Dot Diagrams} Given a positive rational number $q=P/Q$ with $P$ and $Q$ relatively prime, we begin with what we will call the \emph{box-dot template} for $q$.  We first consider the rectangle $\R_q=[0,P]\times[0,Q]$ in $\rr^2$, shown in Figure~\ref{egbdtemplate} for $q=5/3$.  The box-dot template for $q$ will consist of a collection $\B\cup\D$ of marked points (\emph{boxes} and \emph{dots}) in $\R_q$, given by
\begin{align*}
\B&=\{(i-\textstyle{\frac{1}{2}},j)|i\in\{1,2,\ldots,P\},\, j\in\{0,1,\ldots,Q\}\},\\
\D&=\{(i,j-\textstyle{\frac{1}{2}})|i\in\{0,1,\ldots,P\},\, j\in\{1,2,\ldots, Q\}\}.
\end{align*}

\begin{figure}[ht]
\begin{center}
\resizebox{0.85\textwidth}{!}{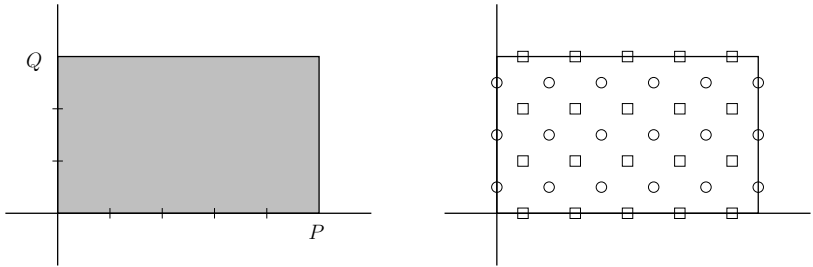}  
\caption[The box-dot template.]{The rectangle $\R_q$ (left) and the box-dot template $\B\cup\D$ (right) for $q=5/3$.}
\label{egbdtemplate}
\end{center}
\end{figure}

Naturally, we will refer to the intersection of the template with either the interior or the boundary of the rectangle $\R_q$ as the \emph{interior points} or the \emph{boundary points} of the template, respectively.

We next describe how to construct a particular box-dot diagram, called the \emph{standard position} box-dot diagram for $q$.  This diagram is obtained by way of a subdivision of the rectangle $\R_q$ into squares of maximal area.  To begin the subdivision, let $\r_1=\R_q$.  If $P \geq Q$ we subdivide $\r_1$ into the union of a collection of squares of dimension $r_1=Q$, positioned at the right end of $\r_1$, with a \emph{remainder rectangle} $\r_2$ at the left.  Let $q_1$ denote the number of squares produced in this first subdivision.  We may assume the width $r_2$ of the rectangle $\r_2$ is strictly less than $r_1$, since otherwise we may continue to subdivide squares of dimension $r_1$ from $\r_2$.  If $Q=1$ then $q=P$ is an integer, in which case $q_1=P$ and $r_2=0$, and our process is complete.  If $Q > P$ then there are no squares of dimension $Q$ contained in $\r_1=\R_q$, so $q_1=0$ and $r_2=P$, and we continue. 

Since $r_2<r_1$, we then subdivide $\r_2$ into the union of a collection of $q_2$ squares of dimension $r_2$ with another remainder rectangle $\r_3$.  Here we position these $q_2$ squares at the bottom end of $\r_2$.  As above, we may assume that the height $r_3$ of the rectangle $R_3$ is strictly less than $r_2$.

We then alternate between the previous two steps, subdividing squares from the right or bottom ends of the remainder rectangle produced at the previous stage of our subdivision.  By this process, we obtain a sequence of nested rectangles $\{\r_j\}$.  Each rectangle $\r_j$ contains $q_j$ squares of dimension $r_{j}$ positioned at the right or bottom of $\r_j$ if $j$ is odd or even, respectively, together with the rectangle $\r_{j+1}$.  For $j>1$ the dimensions of $\r_j$ are $r_j\times r_{j-1}$ for $j$ even, or $r_{j-1}\times r_j$ for $j$ odd.  By construction we then have $0 \leq r_{j} < r_{j-1}$, and so we obtain a decreasing sequence of dimensions $r_1=Q,r_2,r_3,\ldots$, all of which are positive integers.  This sequence must therefore eventually terminate at some value $r_{n}$.  In particular, the final set of $q_n$ squares will either be aligned horizontally if $n$ is odd, or vertically if $n$ is even.  Moreover, by construction, the dimensions $r_j$ and counts $q_j$ of squares satisfy the following system of equations:

\begin{minipage}{\textwidth}
\begin{equation}
\left\{\begin{aligned}
P &= q_1r_1 + r_2 \\
r_1 &= q_2r_2 + r_3 \\
r_2 &= q_3r_3 + r_4 \\
r_3 &= q_4r_4 + r_5 \\
\vdots & \\
r_{n-3} &= q_{n-2}r_{n-2} + r_{n-1} \\
r_{n-2} &= q_{n-1}r_{n-1} + r_{n} \\
r_{n-1} &= q_nr_n \\ 
\end{aligned}\right.
\label{euclidalgorithm}
\end{equation}
\end{minipage}

Note that since $r_n$ divides $r_{n-1}$, it must also divide $r_{n-2}$, and so it must also divide $r_{n-3}$, and so, inductively, it must also divide both $P$ and $Q$.  Since $P$ and $Q$ were assumed to be relatively prime, we have that $r_n=1$.  We remark here that the subdivision we have described can indeed be seen as a geometric interpretation of the Euclidean algorithm.

From this construction we obtain a vector $(q_n,q_{n-1},\ldots,q_1)$ whose components are the counts $q_j$ of squares of dimension $r_{j}$ in our subdivision.  Further, it can be seen from Equations~\eqref{euclidalgorithm} that we may recover the rational number $q$ from this vector via the continued fraction 
\[ q = q_1 + \cfrac{1}{q_2+\cfrac{1}{\ddots+\cfrac{1}{q_n}}}.\]

\begin{figure}[hb]
\begin{center}
\resizebox{\textwidth}{!}{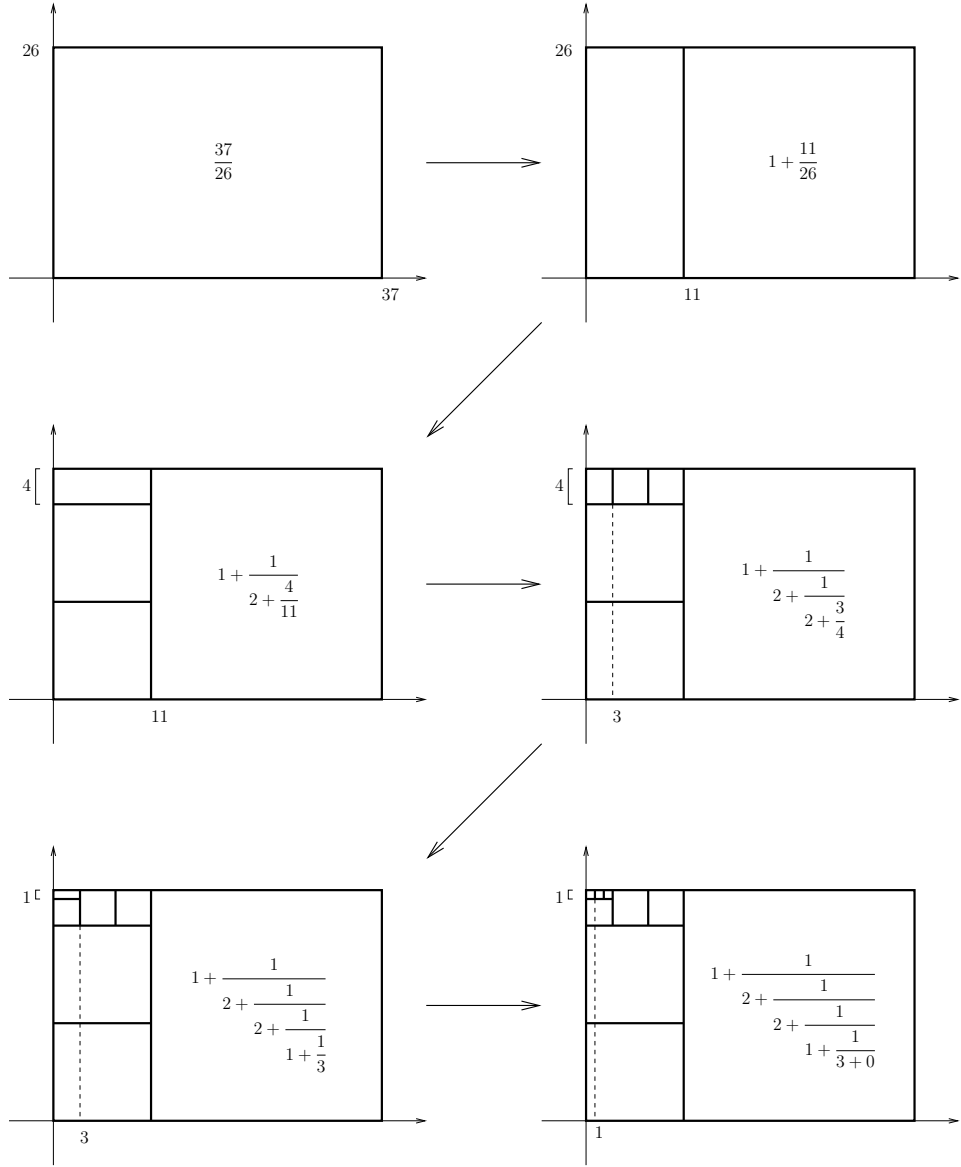}  
\caption[Subdivision by squares for $q=37/26$.]{Subdivision by squares of $R_q$ for $q=37/26\sim(3,1,2,2,1)$.}
\label{egsquaring}
\end{center}
\end{figure} 

Figure~\ref{egsquaring} illustrates an example of the subdivision process described above, while simultaneously expanding the corresponding continued fraction to illustrate the various roles of our notation.  The reason for the term ``remainder rectangle'' should also be evident in this example.

\begin{regcontfrac}
\label{regcontfrac}
The square counts $q_j$ satisfy the following properties:
\begin{enumerate}
\item $q_n \geq 2$ if $n>1$, or if $n=1$ and $q > 1$
\item $q_j \geq 1$ if $1 < j < n$
\item $q_1 \geq 0$, with $q_1 = 0$ iff $q < 1$
\end{enumerate}
\end{regcontfrac}

\begin{proof}
We first remark that $q_j \geq 0$ for $1 \leq j \leq n$, as these represent counts of squares in the subdivision.  If $n>1$, we have from Equations~\eqref{euclidalgorithm} that $r_n < r_{n-1}=q_nr_n$.  Thus $1<q_n$ since $r_n=1$.  If $n=1$ then $q$ is a positive integer, so $q_1 > 1$ if $q > 1$ since $q_1=q$.  This completes the proof of (1).  

Next, suppose $q_j=0$ for some $j$ with $1<j<n$.  Then by Equations~\eqref{euclidalgorithm} we have $r_{j-1} = q_j r_j + r_{j+1} = r_{j+1}$.  However $r_{j-1} > r_{j+1}$ by construction, providing a contradiction.  Thus $q_j\neq 0$ for $1<j<n$, proving (2).

Finally, (3) follows directly from our construction.  We defined $q_1$ to be $0$ if $P<Q$ and, conversely, if $q_1=0$ then no squares of dimension $Q$ are contained in $P$, implying that $P<Q$.
\end{proof}

The statement of Proposition~\ref{regcontfrac} does allow the possibility that $q=1$, with corresponding vector $(q_1)=(1)$.  More importantly, recall from Section~\ref{sec:tangleconstruction} that conditions (1--3) of Proposition~\ref{regcontfrac} are precisely the conditions required for a continued fraction expansion for $q>0$ to be regular.  Thus, by counting squares, such a subdivision can be used to recover the unique regular continued fraction representation of any positive rational number.  We now define the \emph{standard position box-dot diagram} $\boxdot_q$ of a positive rational number $q$ to be the intersection of the interior points of the box-dot template for $q$ with the edges of the squares in the above subdivision, as illustrated in Figure~\ref{egbddiagram}.

\begin{figure}[ht]
\begin{center}
\includegraphics[width=0.7\textwidth]{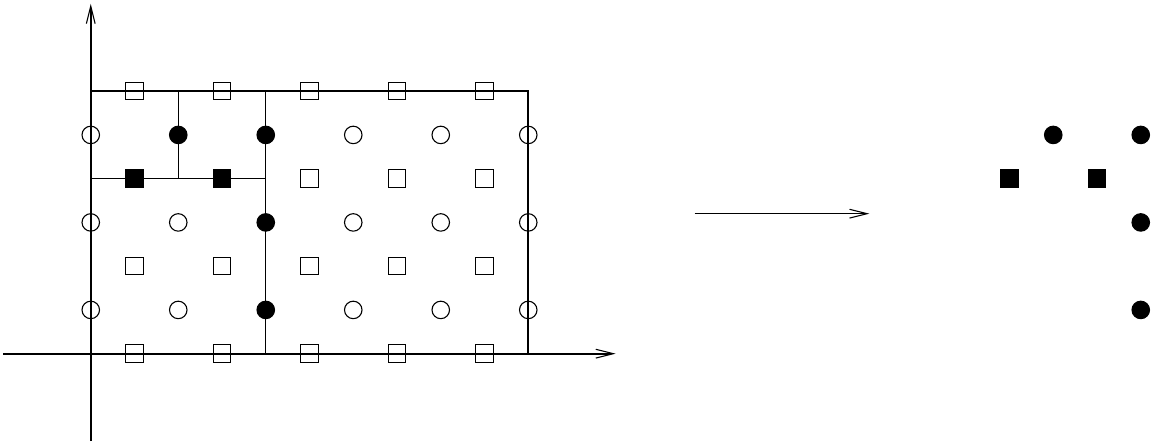}
\caption[The standard position box-dot diagram $\boxdot_{5/3}$.]{The standard position box-dot diagram $\boxdot_{5/3}$ (right).}
\label{egbddiagram}
\end{center}
\end{figure}

\begin{bduniqueness}
For any rational number $q>0$, the corresponding standard position box-dot diagram $\boxdot_q$ is unique.
\end{bduniqueness}

\begin{proof}
In brief, this proposition follows from the uniqueness of the regular continued fraction expansion for $q$.  Suppose we have two standard position box-dot diagrams for $q$.  Then the numbers of squares in each, in order of increasing area, give two vectors $(q_n,q_{n-1},\ldots,q_1)$ and $(q'_m,q'_{m-1},\ldots, q'_1)$.  Since each of these vectors must correspond to a regular continued fraction expansion for $q$, as verified by Equations~\eqref{euclidalgorithm}, we must have that $m=n$ and $q'_j=q_j$ for $1\leq j \leq n$, as the regular continued fraction for $q$ is unique.  This in turn shows that the subdivisions associated to each of the box-dot diagrams consist of the same numbers of squares of the same areas, and so the diagrams must be identical if both are in standard position.
\end{proof}

When we intersect the subdivision of $\R_q$ with the interior of the box-dot template corresponding to $q$, any vertical edges of squares in the subdivision occur at integral values of $x$, and so will intersect a collection of dots in the template.  Similarly, horizontal edges occur at integral values of $y$, and so will intersect a collection of boxes.  We next describe the box-dot diagram $\boxdot_q$ explicitly as a union $\b \cup \d$ (\emph{boxes} and \emph{dots}) of unions $\d =  \bigcup_{k \textrm{ odd}}\d_k$ and $\b = \bigcup_{k \textrm{ even}}\b_k$.  In particular, for $q=P/Q \geq 1$, the sets $\b_k$ for $k$ odd and $\d_k$ for $k$ even,  for $1\leq k < n$ are given by
\begin{align*}
\d_k &= \left\{(i,j-\textstyle{\frac{1}{2}})\left|\begin{array}{l} i\in\{r_{k-1}-q_k r_k, r_{k-1}-(q_k-1) r_k, \ldots, r_{k-1}-r_k \},\\ j\in\{Q-r_k+1,Q-r_k+2,\ldots, Q\}\\ \end{array}\right.\right\},\\
\b_k &= \left\{(i-\textstyle{\frac{1}{2}},j)\left|\begin{array}{l} i\in\{1,\ldots,r_k\},\\ j\in\{Q-r_{k-1}+r_k, Q-r_{k-1}+2r_k,\ldots,Q-r_{k-1}+q_k r_k \}\end{array}\right.\right\},
\end{align*}
where we assume $r_0=P$ for $\d_1$.  If instead $q<1$, then $\d_1=\emptyset$ and the sets $\d_k$ for $k$ odd and $\b_k$ for k even, remain the same for $1<k<n$, assuming $r_2=P$ for $\d_2$.  In either case, to ensure that our diagram consists only of interior points of the template, the last set, either $\d_n$ or $\b_n$ for $n$ respectively odd or even, is given explicitly as
\begin{align*}
\d_n &= \left\{(i,Q-\textstyle{\frac{1}{2}})| i\in\{1, 2, \ldots, q_n-1\} \right\},\\
\b_n &= \left\{(\textstyle{\frac{1}{2}},j)| j\in\{Q-q_n+1, Q-q_n+2,\ldots,Q-1\}.\right\}.
\end{align*}

\begin{bdcount} 
\label{bdcount}
The standard position box-dot diagram $\boxdot_q$ for $q=P/Q$ contains exactly $P-1$ dots and $Q-1$ boxes.
\end{bdcount}

\begin{proof}
By the above descriptions, we have
\begin{align*}
|\d_1|&=q_1 Q = P-r_2\textrm{ if }q>1\textrm{ or }|\d_1|=0=P-r_2\textrm{ if }q<1,\\
|\d_k|&=q_k r_k\textrm{ for }k\textrm{ odd, }1<k<n,\\
|\b_k|&=r_k q_k\textrm{ for }k\textrm{ even, }1<k<n,\\
|\d_n|&=|\b_n|=q_n-1.
\end{align*}
Assuming that $n$ is odd, the total number of dots is then given by
\begin{align*}
|\d| &= |\d_1| + |\d_3| + |\d_5| + \ldots + |\d_{n-2}| + |\d_n|\\
&= |\d_1| + q_3r_3 + q_5r_5 + \ldots + q_{n-2}r_{n-2} + q_n - 1\\
&= (P-r_2) + (r_2-r_4) + (r_4-r_6) + \ldots + (r_{n-3}-r_{n-1}) + r_{n-1} - 1\\
&= P-1,
\end{align*}
while the total number of boxes is given by
\begin{align*}
|\b| &= |\b_2| + |\b_4| + |\b_6| + \ldots + |\b_{n-3}| + |\b_{n-1}|\\
&= q_2r_2 + q_4r_4 + q_6r_6 + \ldots + q_{n-3}r_{n-3} + q_{n-1}r_{n-1}\\
&= (Q-r_3) + (r_3-r_5) + (r_5-r_7) + \ldots + (r_{n-4}-r_{n-2}) + (r_{n-2}-r_n)\\
&= Q-1.
\end{align*}
For $n$ even, the proof follows from similar computations.
\end{proof}

By Proposition~\ref{bdcount}, we can recover the rational number $q$ from a box-dot diagram via the formula $P/Q=\left(|\d|+1\right)/\left(|\b|+1\right)$.  Thus we lose no information by suppressing the boundary points as we have, since the corresponding template can always be reconstructed from the interior diagram by first recovering $P/Q$.  

\subsection{F-moves on Box-Dot Diagrams}
\label{sec:fmoves}
For the standard position box-dot diagram, we required that the squares in our subdivision be positioned exclusively at the right and bottom ends of each remainder rectangle.  If instead we allow squares of maximal area to be positioned at either end, or both ends, of the remainder rectangles, we obtain what we will call a \emph{free subdivision} of $\R_q$ by squares.  If $q$ corresponds to the vector $(q_n,q_{n-1},\ldots,q_1)$, then for any $j$ with $1 \leq j < n$ we may position $f_j$ of the $q_j$ squares of dimension $r_j$ at the left or top end, accordingly, of the remainder rectangle $\r_j$, for any $f_j$ with $0\leq f_j \leq q_j$.  Since the $q_n$ squares of dimension one uniquely subdivide the remainder rectangle $\r_n$, there is thus a total of $\prod_{j=1}^{n-1}(q_j+1)$ distinct free subdivisions of the rectangle $\R_q$.  We define a general \emph{box-dot diagram} for $q$ to be the intersection of any of these free subdivisions of $\R_q$ with the box-dot template for $q$.

We now define a set of operations, called \emph{f-moves}, that we can perform on the standard position box-dot diagram $\boxdot_q$ to generate any other box-dot diagram for $q$.  The f-moves on $\boxdot_q$ are in one-to-one correspondence with vectors $f=(0, f_{n-1},f_{n-2},\ldots, f_1)$ whose components satisfy $0 \leq f_j \leq q_j$ for $0 \leq j \leq n-1$.  The image of the f-move, denoted $\boxdot_{q^f}$, will be the intersection of the box-dot template for $q$ with a particular free subdivision of $R_q$ constructed according to the vector $f$.

  Suppose first that $f=(0,\ldots,0,1,0,\ldots,0)$, with $f_j=1$ for some $j \neq 1$ and all other components equal to zero.  The diagram $\boxdot_{q^f}$ will agree with $\boxdot_q$ over all squares corresponding to components $q_{j-1}, q_{j-2},\ldots,q_{1}$.  The remainder rectangle $\r_j$ is then modified by first reflecting the remainder rectangle $\r_{j+1}$ along its vertical or horizontal bisector, depending on whether $j$ is even or odd, respectively.  The reflected rectangle $\r_{j+1}$ is then transposed with the adjacent square in $\r_j$ lying either below or to the right of $\r_{j+1}$, again according to whether $j$ is even or odd, respectively.  These operations are illustrated in Figure~\ref{fmoves}.

\begin{figure}[ht]
\begin{center}
\resizebox{0.5\textwidth}{!}{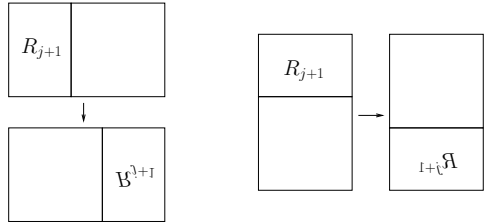}
\caption[$f$-moves for box-dot diagrams.]{The two types of basic $f$-moves for box-dot diagrams, corresponding to $j$ odd (left) and $j$ even (right).  Shown in each case is the remainder rectangle $\r_j$.}
\label{fmoves}
\end{center}
\end{figure}

To construct the image of the $f$-move corresponding to a general vector $f=(0,f_{n-1},\ldots,f_1)$, we begin first with the portion of $\boxdot_q$ corresponding to the component $q_{n-1}$, applying the above process to the rectangle $\r_{n-1}$ a total of $f_{n-1}$ times.  We then apply this process to the rectangle $\r_{n-2}$ a total of $f_{n-2}$ times, then to $\r_{n-3}$ a total of $f_{n-3}$ times, and so on.  Iterating this process through all components of the vector corresponding to $q$ will result in the desired diagram $\boxdot_{q^f}$.  

\begin{figure}[ht]
\begin{center}
\includegraphics[width=0.65\textwidth]{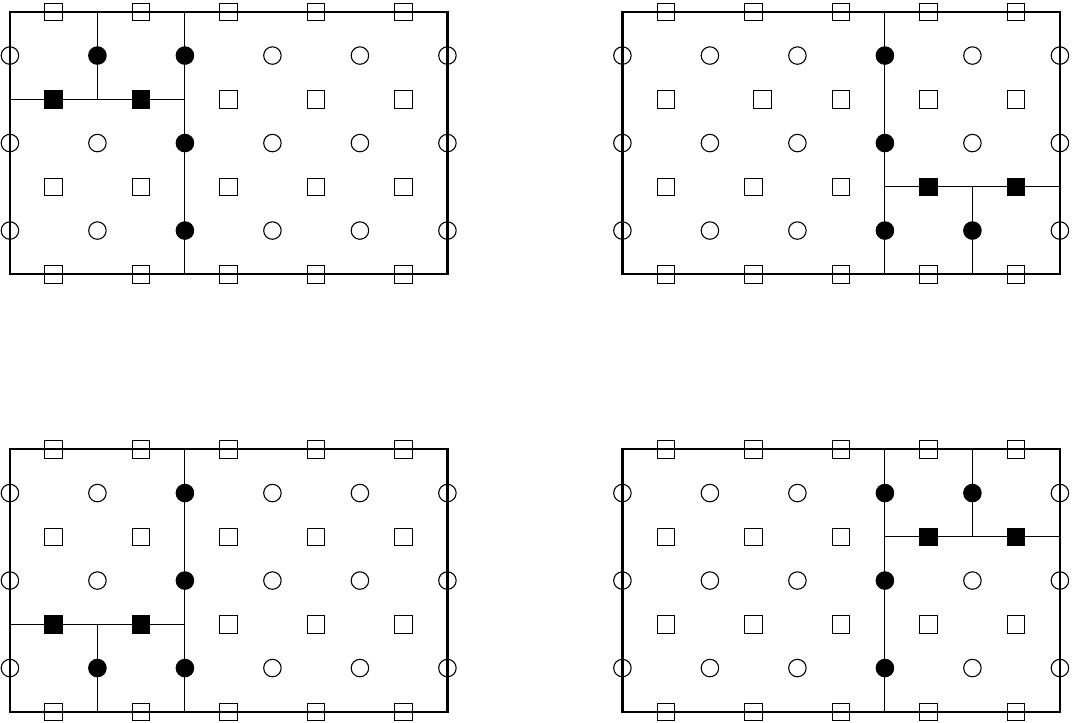}
\caption[The box-dot diagrams $\boxdot_{(2,1,1)}$, $\boxdot_{(2,1^1,1)}$, $\boxdot_{(2,1^1,1^1)}$, and $\boxdot_{(2,1,1^1)}$.] {Counterclockwise from top left: the box-dot diagrams, with templates, $\boxdot_{(2,1,1)}$, $\boxdot_{(2,1^1,1)}$, $\boxdot_{(2,1^1,1^1)}$, and $\boxdot_{(2,1,1^1)}$.}
\label{egbddfmoves}
\end{center}
\end{figure}

Figure~\ref{egbddfmoves} illustrates all of the possible box-dot diagrams for $q=5/3$.  As in Section~\ref{sec:topflypes}, when we wish to specify a particular f-move $\boxdot_{q^f}$ explicitly relative to the components of the vector corresponding to $q$, we will employ the notation $(q_n,q_{n-1}^{f_{n-1}},\ldots,q_1^{f_1})$ in place of $q^f$, omitting superscripts corresponding to $f_j$'s which are zero.

\section{Applications of Box-Dot Diagrams}
\label{sec:boxdotconstructions}
Despite their elementary construction, the box-dot diagrams defined in the previous section contain the information necessary to construct all of the topological and contact geometric objects associated to a rational tangle in sections~\ref{sec:rationaltangles} and~\ref{sec:contactgeometry}.  To illustrate these constructions, we include throughout this section a running example with $q=5/3\sim(2,1,1)$.  

\subsection{Topological Rational Tangles}  
\label{sec:boxdottoptangles}
Beginning with the box-dot diagram $\boxdot_q$, we reconstruct the associated box-dot template and subdivision by squares.  We then inscribe each square in the subdivision with a single crossing such that the endpoints of the strands of each crossing coincide with the corners of the corresponding square.  As a matter of convention we require the overarc of each crossing to have a more negative slope than the underarc.  When two strands from distinct crossings intersect at adjacent corners of the corresponding squares, we smooth the resulting intersection to create a single arc which does not intersect the boundary of the remainder rectangle containing the adjacent squares.  Doing this at each adjacent pair of corners results in a planar projection of a tangle, as in Figure~\ref{egbddtangle}.

\begin{figure}[ht]
\begin{center}
\includegraphics[width=0.8\textwidth]{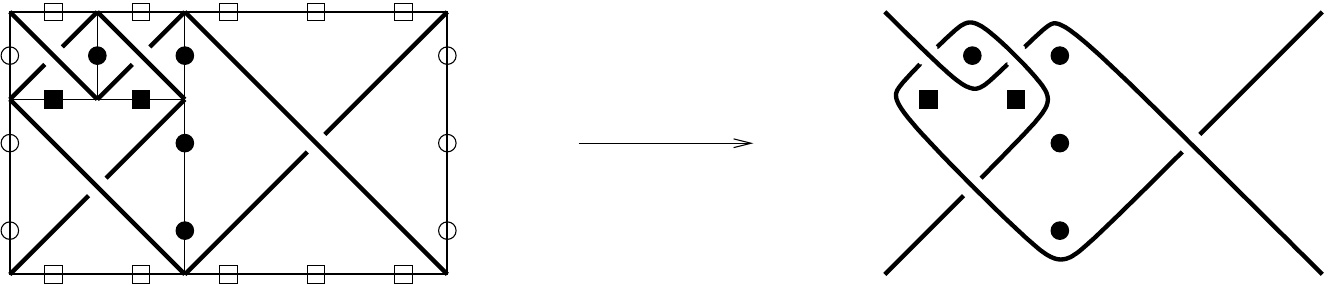}
\caption[Constructing $G_{5/3}$ from $\boxdot_{5/3}$.]{Constructing $G_{(2,1,1)}$ from $\boxdot_{(2,1,1)}$.}
\label{egbddtangle}
\end{center}
\end{figure}

\begin{bddtangle}
The tangle associated to the standard position box-dot diagram $\boxdot_{q}$ is the rational tangle $G_q$.
\end{bddtangle}

\begin{proof}
The free subdivision by squares corresponding to $\boxdot_q$ yields a unique vector $(q_n,q_{n-1},\ldots,q_1)$ for $q$.  By the construction of $\boxdot_q$, in the case that $n$ is odd, this subdivision will contain a row of $q_n$ squares of dimension one arranged in a row at the top left corner of the rectangle $\R_q$.  Below this will be a column of $q_{n-1}$ squares, followed by an alternating sequence of rows of squares on the right and columns of squares below according to the vector for $q$.  Inscribing each square with a single crossing yields a tangle that can be described as an alternating sequence of horizontal and vertical twists in the same order and number as in the construction of $G_q$ given in Section~\ref{sec:tangleconstruction}.  As the crossings obey the same convention in both constructions, the tangle arising from $\boxdot_q$ is, in fact, $G_q$.  

The proof in the case where $n$ is even differs only in that the subdivision instead contains a column of $q_n$ squares of dimension one arranged vertically in the top right corner of $\r_q$, corresponding to the initial vertical twisting of an even-length tangle.
\end{proof} 

This construction also provides a correspondence between diagrams $\boxdot_{q^f}$ and planar diagrams of flyped tangles $G_{q^f}$.  The reflection and transposition steps used to describe f-moves in Section~\ref{sec:fmoves} correspond to the rotation of the sphere described in Section~\ref{sec:topflypes} which resulted in a transposed crossing. 

\subsection{Legendrian Rational Tangles}
\label{sec:boxdotlegtangles}
By our conventions, the crossings in the projection of $G_q$ arising from $\boxdot_q$ automatically conform to the first criterion for a front projection.  To establish the second, we need only replace any vertical tangencies with cusps.  Such tangencies occur precisely at vertically adjacent pairs of squares in the corresponding subdivision.  As discussed in Section~\ref{sec:legflypes}, we will modify these points so that the resulting cusps have horizontal tangents.

The resulting front projection may then be lifted via the Legendrian lifting map of Section~\ref{sec:legendrianization} to a Legendrian embedding $\G_q$ of the topological tangle $G_q$.  Similarly, by Legendrianizing the projection of $G_{q^f}$ obtained from any diagram $\boxdot_{q^f}$ we produce a correspondence between these diagrams and the Legendrian flypes $\G_{q^f}$ of $\G_q$ defined in Section~\ref{sec:legflypes}.  This correspondence is illustrated in Figure~\ref{egbddflypes}.  We will refer to any tangle $\G_{q^f} \subset \rr^3$ constructed from $\boxdot_{q^f}$ in this manner as a \emph{regular Legendrian rational tangle}, reflecting its connection to a box-dot diagram, and the regular continued fraction encoded therein. 

\begin{figure}[ht]
\begin{center}
\includegraphics[width=0.7\textwidth]{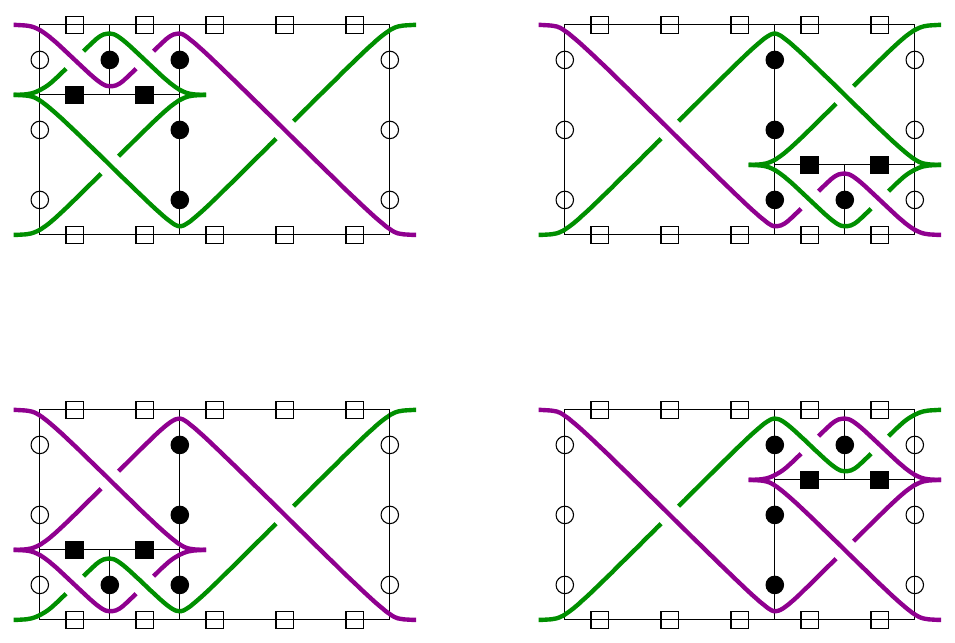}
\caption[Correspondence between $f$-moves on $\boxdot_{5/3}$ and Legendrian flypes of $\G_{5/3}$.]{Counterclockwise from top left: box-dot diagrams (with subdivisions) for $\G_{(2,1,1)}$, $\G_{(2,1^1,1)}$, $\G_{(2,1^1,1^1)}$ and $\G_{(2,1,1^1)}$.}
\label{egbddflypes}
\end{center}
\end{figure}

\subsection{The Ambient Ball $\boldsymbol{B_q}$}
\label{sec:ambientball}
Recall from Section~\ref{sec:tangleconstruction} that rational tangles are properly embedded in a 3-dimensional ball, and that tangle isotopies are isotopies of this ambient ball which restrict to the identity along its boundary.  We show now that we can construct a particular ambient ball for a regular Legendrian rational tangle from its box-dot diagram.

For a positive rational number $q=P/Q$, let $\plB_q$ denote the rectangular box in $\rr^3$ given by $\plB_q=[0,P]\times[-1,1]\times[0,Q]$, so that the rectangle $\R_q$ from Section~\ref{sec:boxdotdiagrams} is then the projection of $\plB_q$ into the $xz$-plane.  By smoothing $\plB_q$ along a small neighborhood of its edges we obtain a smooth object $B_q$, topologically a 3-ball, that agrees with $\plB_q$ on all but this neighborhood.  We will denote by $p_1$, $p_2$, $p_3$, and $p_4$ the images on $\partial B_q$ of the points $(0,-1,Q)$, $(P,1,Q)$, $(0,1,0)$, and $(P,-1,0)$ on $\partial \plB_q$ after this smoothing process.  These points will serve as the endpoints of the strands of $\G_q$.

As discussed in Section~\ref{sec:characteristicfoliations}, the characteristic foliation of a surface in $(\rr^3,\xi_{std})$ is obtained from the intersection of its tangent planes with the contact planes at every point.  The characteristic foliation of $\partial \plB_q$ can be assembled from the characteristic foliations of the individual faces, which are easily seen to be as illustrated in Figure~\ref{folball} using the basis of $\xi$ given in Section~\ref{sec:contactgeometry}.  In smoothing $\plB_q$ to $B_q$ we then obtain a smooth, singular foliation on $\partial B_q$ which is very similar to the characteristic foliation of the unit sphere in $(\rr^3,\xi_{std})$.  In fact, the only fundamental difference between these foliations is that the elliptic points at the poles of the sphere are replaced with what are known as \emph{generalized elliptics}---curves of singularities---in $(\partial B_q)_{\xi}$ where $\partial B_q$ meets the line $x=0$ in each of the planes $z=0$ and $z=Q$.  

\begin{figure}[ht]
\begin{center}
\includegraphics[width=0.8\textwidth]{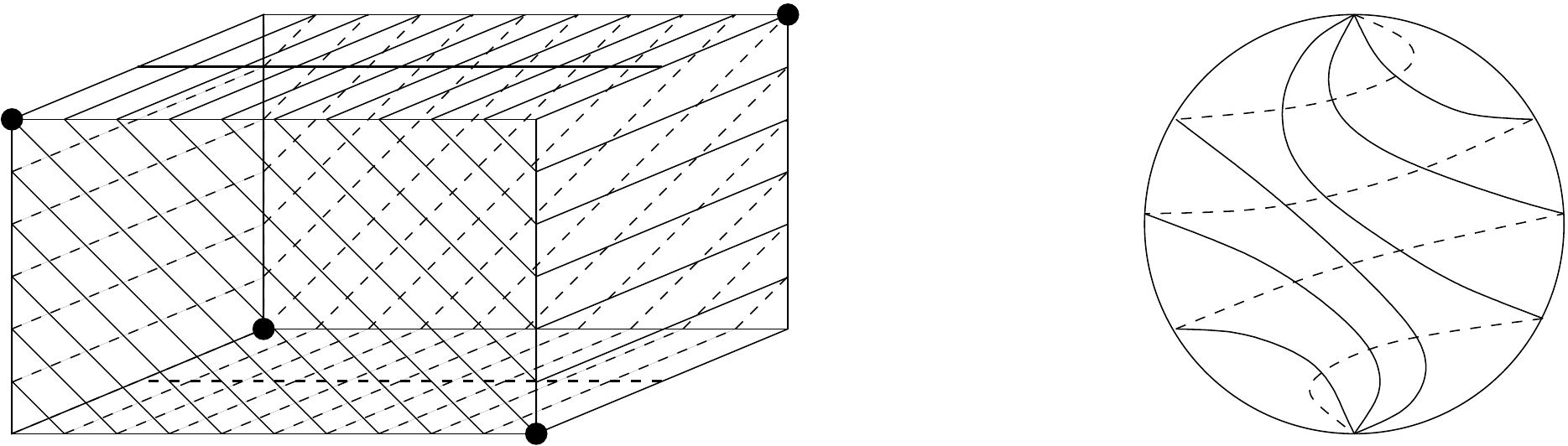}
\caption[The ambient ball $B_q$.]{The characteristic foliation and marked points on $\plB_q$ (left) and the characteristic foliation of the unit sphere in $(\rr^3,\xi_{std})$, as seen in~\cite{E03} (right).}
\label{folball}
\end{center}
\end{figure}

Up to some adjustment accounting for the smoothing process, the restriction that Legendrian tangle isotopies may not leave the ambient ball $B_q$ reduces to requiring that transverse isotopies of the front projection of $\G_q$ neither stray outside of the boundary of $\R_q$, nor ever produce arcs with a slope of magnitude greater than or equal to one.  The former of these conditions keeps isotopies contained in $B_q$ in the $x$ and $z$ directions in the obvious way.  The latter prevents the corresponding Legendrian curves from leaving $B_q$ in the $y$-direction.

While this embedding of $B_q$ and the tangles it contains may seem awfully specific, we remark that this primarily serves to ease the exposition, as we may rescale $B_q$ in the $x$ and $z$ directions freely.  If the $x$ and $z$ dimensions are scaled by factors $a$ and $b$, respectively, then the tangle isotopy conditions described above remain valid, provided we also scale the $y$ dimension by a factor of $b/a$.  We may also translate the ball $B_q$ (and all its contents) freely in the $x$ and $z$ directions, as the contact structure is preserved under such translations.  We remark here that we only employed the bounding rectangle $\R_q$ of the box-dot diagram in this construction, so any diagram $\boxdot_{q^f}$ will yield the same ambient ball $B_q$.

\subsection{The Legendrian Unknot $\boldsymbol{K_q}$}
\label{sec:unknotKq}
Using the boundary points of the box-dot template for $\boxdot_q$, we now construct a Legendrian unknot in the characteristic foliation of $B_q$ which will serve as the boundary of a compressing disc for $B_q-\G_q$ in Section~\ref{sec:compressingdisc}.  By construction, there are a total of $2P$ boxes and $2Q$ dots in the boundary of the box-dot template.  The set of $P+Q$ segments of slope $-1$ in $\rr^2$ whose endpoints are incident with these boundary points together with the $P+Q$ segments of slope $1$ with the same endpoints yield a piecewise linear planar projection of an unknot, as Proposition~\ref{unknotKq} will establish.   Again we will adhere to the convention that whenever two of these segments intersect, we form a crossing such that the segment of negative slope becomes the overarc.  Applying the Legendrian lifting map to these individual segments produces a disjoint collection of segments in the faces $y=\pm 1$ of $\plB_q$.  Corresponding to each of the boundary points we include the segment with the same $x$- and $z$-coordinates, and with $y$-coordinate ranging between $-1$ and $1$, each connecting a lifted segment in the plane $y=1$ to a lifted segment in the plane $y=-1$.  All of these segments lie in the characteristic foliation of $\plB_q$.  We will denote the union of these segments as $\plK_q$.  

\begin{unknotKq}
\label{unknotKq} 
Topologically, $\plK_q$ is an unknot. 
\end{unknotKq}

The proof will be carried out via the projection of $\plK_q$ into the $xz$-plane, with the assistance of the following lemma.  In order to state this lemma simply, let $m$ be a positive integer, and let $S_m$ denote the square $[0,m]\times[0,m]$ in $\rr^2$ with $4m$ marked points along its boundary at coordinates $(0,j-1/2)$, $(j-1/2,0)$, $(m,j-1/2)$, and $(j-1/2,m)$, for $j\in\{1,2,\ldots,m\}$.  As in the above construction, we adjoin these marked points with a collection of $4m$ segments of slope $\pm 1$ such that the segments of slope $-1$ become the overarc in any crossings that occur.  Let $U_m$ denote the union of these segments.  Figure~\ref{squareunlink} illustrates this construction for $m=5$.

\begin{figure}[ht]
\begin{center}
\includegraphics[width=0.3\textwidth]{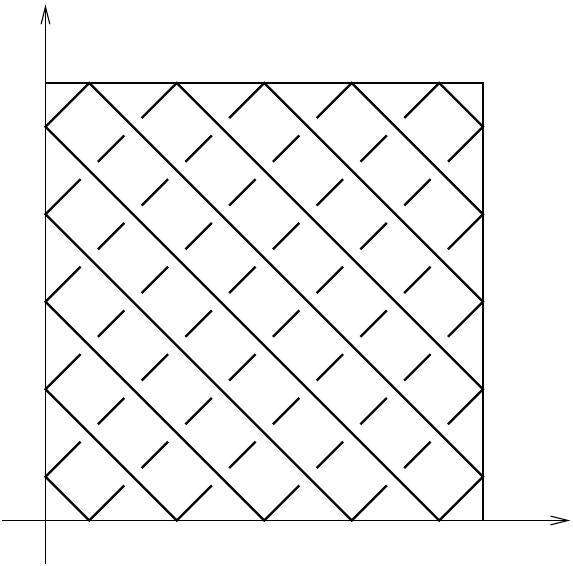}
\caption[The Unlink $U_5$.]{The unlink $U_5$, consisting of five unknots.}
\label{squareunlink}
\end{center}
\end{figure}

\begin{squareunknots}
\label{squareunknots}
$U_m$ is a planar projection of a piecewise linear Legendrian unlink, consisting of $m$ unknots.
\end{squareunknots}

\begin{proof}
The corresponding Legendrian link is the union of the lifts of the given segments, together with the suspensions of the marked points as above.  In the planar projection $U_m$, each marked point $(0,j-1/2)$ is connected by a segment of slope $-1$ to the marked point $(j-1/2,0)$.  This point is connected by a segment of slope $1$ to the marked point $(m,m-j+1/2)$, which is then connected by a segment of slope $-1$ to the marked point $(m-j+1/2,m)$, which then connects to the initial point by a segment of slope $1$.  Since none of these segments cross over any of the others in this loop, it must correspond to an unknot.  Thus, $U_m$ is a projection of $m$ unknots, one corresponding to each initial point $(0,j-1/2)$.  By our crossing convention, the unknot corresponding to $(0,m-1/2)$ lies on top of all of the other unknots in the planar projection of the link $U_m$, and so can be shrunk to a point in the complement of the other components in $\rr^3$.  The unknot corresponding to $(0,m-3/2)$ now lies on top of the remaining $m-2$ components, and so again we can shrink this unknot to a point in the complement of the other components.  Proceeding inductively, it is thus possible to shrink each of the $m$ unknots in $U_m$ to a point in the complement of the others, proving that $U_m$ is a projection of an unlink. 
\end{proof}

\begin{proof}[Proof of Proposition~\ref{unknotKq}]  
We begin with the diagram $\boxdot_q$, and reconstruct the boundary points and associated subdivision of $\R_q$.  By construction, each of the boxes or dots in $\boxdot_q$ will occur at one of $P+Q-2$ crossings of the projection of $\plK_q$.  For each such marked point $(x,z)$, we attach to $\plK_q$ the line segment between $(x,-1,z)$ and $(x,1,z)$, allowing us to view each square of the subdivision as a decorated square $S_m$ with $m=r_j$ for some $j$.  If $q>1$, we begin with the right-most of the $q_1$ squares of dimension $r_1$.  Lemma~\ref{squareunknots} then implies that this square contains $r_1$ unknots, each of which can be shrunk to a point in the complement of the others.  By performing these isotopies, we deform the projection of $\plK_q$ through this first square.  Repeating this process inductively through each of the $q_1$ squares of dimension $r_1$ similarly deforms the projection to the remainder rectangle $\r_2$.  

Continuing on, or beginning here if $q<1$, we then repeat the above process, from the bottom-most square of dimension $r_2$ through each of the $q_2$ such squares in $\r_2$, deforming the projection of $\plK_q$ to the remainder rectangle $\r_3$.  Proceeding inductively through rows and columns of squares, we deform the projection of $\plK_q$ to a point, thus showing that it is a projection of an unknot.  This inductive deformation is illustrated in Figure~\ref{egunknotpf}.
\end{proof}

\begin{figure}[ht]
\begin{center}
\includegraphics[width=0.8\textwidth]{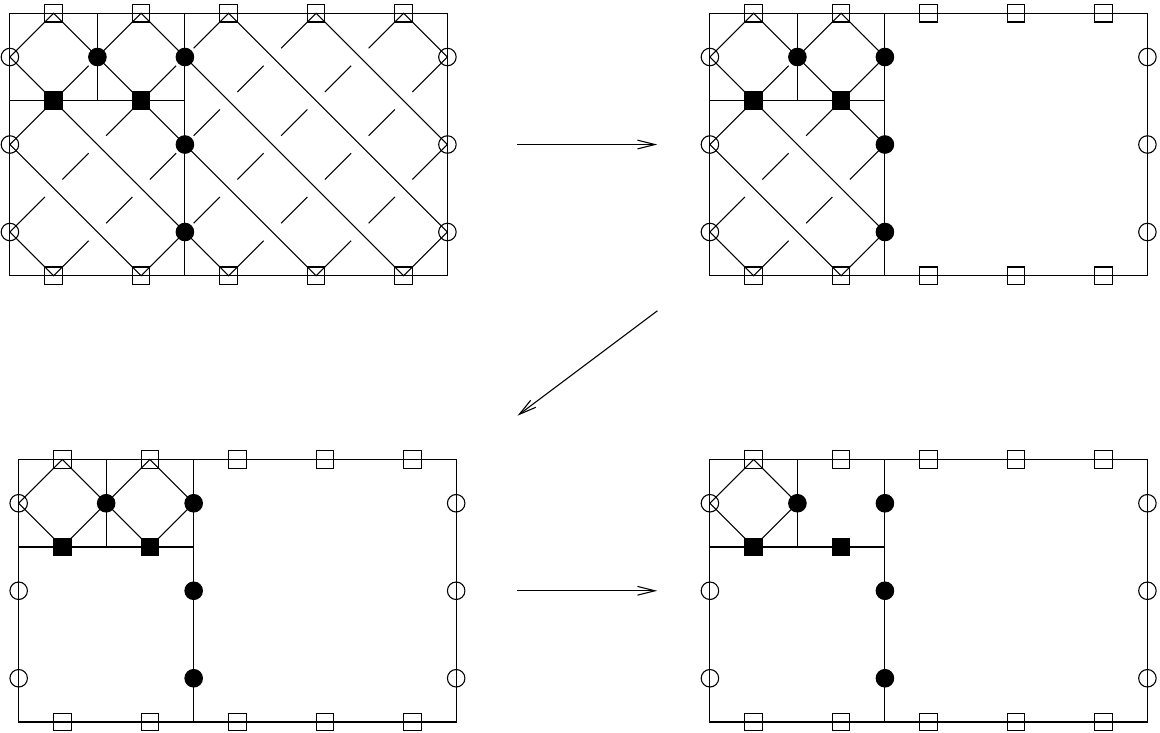}
\caption[Unraveling $\plK_q$.]{Unraveling $\plK_{5/3}$.}
\label{egunknotpf}
\end{center}
\end{figure}

With $\plK_q$ embedded in the characteristic foliation of $\plB_q$, we then apply the smoothing process used to construct $B_q$.  In so doing, we produce a Legendrian unknot $K_q$ in $(B_q)_{\xi}$.  The front projection of $K_q$ can be obtained from the planar projection of $\plK_q$ constructed initially by smoothing any of the corners where two segments of slope $\pm 1$ meet, and by replacing any such smoothings that occur at boundary dots of the corresponding diagram with cusps, as in Figure~\ref{egunknot}.  To ensure that $K_q$ lies on $B_q$, we require that the smoothed corners still pass through the centers of boundary boxes in the template, and that the cusp points occur precisely at the centers of boundary dots\footnote{We generally fail to adhere to this requirement in our figures, though only to make the cusp points more visible.  Conceptually, however, these figures should be viewed as though they satisfy this condition.}.  We remark here that as the unknot $K_q$ is constructed from only the box-dot template, any diagram $\boxdot_{q^f}$ will produce the same unknot $K_q$ under this construction.

\begin{figure}[ht]
\begin{center}
\includegraphics[width=.35\textwidth]{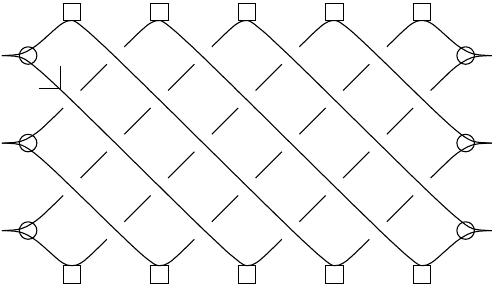} 
\caption[The Legendrian unknot $K_{5/3}$.]{Front projection of the Legendrian unknot $K_{5/3}$.}
\label{egunknot}
\end{center} 
\end{figure}

Proposition~\ref{Kqtbandr} below provides the classical invariants $tb(K_q)$ and $r(K_q)$ (cf.\ Section~\ref{sec:classicalinvariants}) of the Legendrian unknot $K_q$.  In computing these invariants, and throughout the rest of this work, we assume an orientation on $K_q$ induced by traversing this unknot from the top-left dot in the boundary of the corresponding box-dot template along the segment of negative slope leaving it, as in Figure~\ref{egunknot}.  The opposite orientation reverses the sign of the rotation number $r(K_q)$ and preserves the Thurston-Bennequin number $tb(K_q)$.  The proof of this proposition will be deferred until after describing an additional modification to the box-dot template in Section~\ref{sec:compressingdisc}.

\begin{Kqtbandr}
\label{Kqtbandr}
For $q=P/Q$, the Legendrian unknot $K_q$ obtained from $\boxdot_q$ has classical invariants given by:
\begin{align*}
tb(K_q)&=-P\\
r(K_q)&=\left\{\begin{array}{l l} 0 & \textrm{if } P \textrm{ is odd}\\ 1 & \textrm{if } P \textrm{ is even}\\\end{array}\right.
\end{align*}
\end{Kqtbandr}

\subsection{The Compressing Disc $\boldsymbol{D_q}$}
\label{sec:compressingdisc}
We next use the box-dot diagram $\boxdot_q$ to construct a compressing disc $D_q$ for $B_q-\G_q$ (cf.\ Section~\ref{sec:compressingdiscs}).  Specifically, $D_q$ will be a smoothly embedded disc whose boundary is the Legendrian unknot $K_q$ constructed above, and whose interior is contained in the interior of the ball $B_q$.  The complement $B_q-D_q$ will consist of two components, each containing one of the strands of $\G_q$.  We remark again here that while the following construction is described relative to the diagram $\boxdot_q$, it can indeed be performed relative to any box-dot diagram $\boxdot_{q^f}$ to produce a compressing disc for $B_q-\G_{q^f}$.

To begin, we construct from $\boxdot_{q}$ the unknot $K_q$.  As in the proof of Proposition~\ref{unknotKq}, we attach line segments parallel to the $y$-axis at each of the marked points in $\boxdot_{q}$, dividing $K_q$ into a collection of piecewise linear loops.  The isotopy of $K_q$ across a loop then provides a homotopy disc that can be used to fill this loop.  Iterating this process across all of the loops sequentially, smoothing these discs where they meet along the added segments as necessary, we obtain a compressing disc for the ball $B_q$.  However, this disc may not satisfy the additional property that it separates the strands of $\G_q$.  In order to establish this condition, we will appeal to the characteristic foliation of $D_q$.

The foliation of $D_q$ is obtained by gluing together discs with well-known foliations (see~\cite{EF08}).  We view each of the $P+Q-1$ homotopy discs, with piecewise Legendrian boundary made up of arcs of $K_q$ and segments attached at boxes and dots, as bearing the \emph{trivial foliation} shown in Figure~\ref{bddfoliation}.  Where two of these homotopy discs meet along a segment either horizontally (at a dot) or vertically (at a box) we position them so as to bear the \emph{gluing foliations} shown in this figure.  

\begin{figure}[ht]
\begin{center}
\includegraphics[width=0.45\textwidth]{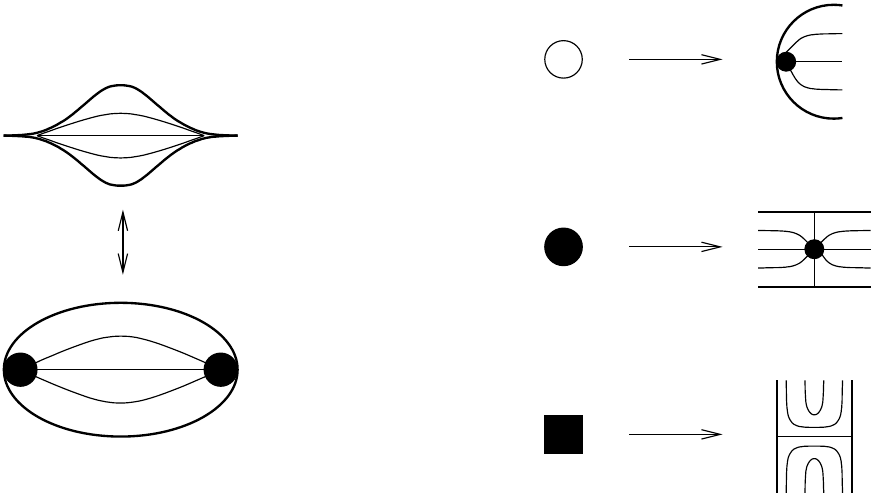}
\caption[Constructing the foliation of $D_q$.]{At left: a trivially foliated disc.  At right: gluing foliations corresponding to boundary dots (top), interior dots (center), and interior boxes (bottom) in a box-dot diagram.}
\label{bddfoliation}
\end{center}
\end{figure}

By construction, the slope of any subarc of the front projection of either strand of $\G_q$ at any point lies in the interval $(-1,1)$.  Moreover, given $\epsilon > 0$ we can arrange that for some $\delta > 0$, the slope of a subarc at any point lying outside of a $\delta$-neighborhood of any cusp or local extreme lies in one of the intervals $(-1,-1+\epsilon)$ or $(1-\epsilon, 1)$.  By positioning the disc accordingly, we then require that any leaf in the foliation of $D_q$ whose front projection intersects that of $\G_q$ has slope in $(-1+\epsilon, 1-\epsilon)$ at the point of intersection.  This guarantees that $D_q$ will separate the strands of the tangle after applying the Legendrian lifting map to each of the leaves in its characteristic foliation, as this will place the corresponding portion of $D_q$ closer to the $xz$-plane than either strand of $\G_q$.  Similar considerations may be applied to the individual foliations of each of the finitely many portions of the disc corresponding to a square in the subdivision.  This is to ensure that $D_q$ contains no point of self-intersection as would appear in the front projection of the foliation as a tangential intersection of leaves.  Figure~\ref{egbddisc} illustrates, via such a projection, an example of a foliated compressing disc constructed in this manner.

\begin{figure}[ht]
\begin{center}
\includegraphics[width=0.35\textwidth]{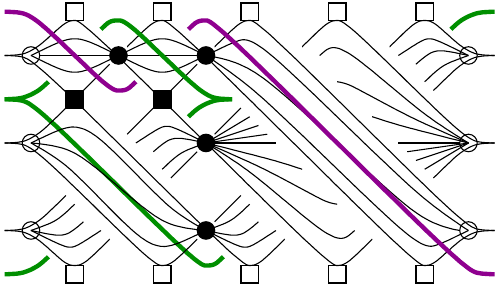}
\caption[The compressing disc $D_{5/3}$.]{The compressing disc $D_{5/3}$ coming from $\boxdot_{5/3}$, together with the corresponding tangle $\G_{5/3}$.}
\label{egbddisc}
\end{center}
\end{figure}

To complete the construction of the foliated disc $D_q$, recall from Section~\ref{sec:characteristicfoliations} that if $D_q$ is oriented then the singularities in $(D_q)_{\xi}$ are either positive or negative, depending on whether the co-orientations of the tangent planes to $D_q$ agree or disagree with those of the contact planes, respectively.  By refining the box-dot template, we can also encode in the corresponding box-dot diagram the signs of singularities of the foliation of $D_q$ relative to a particular choice of orientation on the bounding unknot $K_q$.  Assuming $K_q$ is oriented as at the end of Section~\ref{sec:unknotKq}, shown again in Figure~\ref{egsigntemplate}, we endow $D_q$ with the induced right-handed orientation.  Relative to this choice of orientation we assign a \emph{checkerboard signing} on all of the dots in the box-dot template.  This in turn will induce a particular signing of the dots in any box-dot diagram constructed on this template.  Under this assignment, the sign of any dot will then correspond to the sign of each of the elliptic and boundary hyperbolic singularities of the corresponding gluing foliation.  For the gluing foliations corresponding to each of the boxes in the diagram, we must determine the sign of a pair of boundary hyperbolic points on $K_q$.  The signs of the dots in the diagram induce a similar checkerboard signing of the interior boxes by requiring that a box bear the same sign as the dots adjacent to it along an arc of negative slope in $K_q$.  These signs will correspond to the sign of the hyperbolic point lying along the overstrand at the corresponding crossing of $K_q$, with the hyperbolic point lying along the understrand necessarily bearing the opposite sign, as in Figure~\ref{hyperbolicsigns}.

\begin{figure}[ht]
\begin{center}
\includegraphics[width=0.8\textwidth]{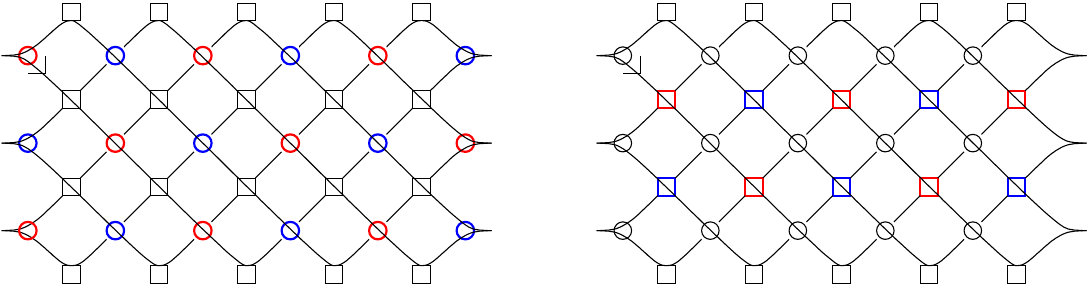}
\caption[Sign conventions for the box-dot template.]{The checkerboard signings on the dots and interior boxes of $\boxdot_q$ induced by the specified orientation.  Markings are colored red or blue to indicate that they are positive or negative, respectively.}
\label{egsigntemplate}
\end{center}
\end{figure}

\begin{figure}[ht]
\begin{center}
\resizebox{.25\textwidth}{!}{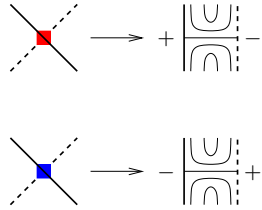}  
\caption[Signs of hyperbolic singularities.]{Sign convention for boxes in the box-dot diagram.  In the foliations shown, the vertical line along the left corresponds to the overstrand of the corresponding crossing of $K_q$.}
\label{hyperbolicsigns}
\end{center}
\end{figure}

With a fixed choice of signs on the dots and interior boxes of a box-dot template, coming from our choice of orientation on $K_q$, we now return to the proof of Proposition~\ref{Kqtbandr}.

\begin{proof}[Proof of Proposition~\ref{Kqtbandr}]
We proceed by direct computation, appealing to Formulas~\eqref{tbformula} and~\eqref{rformula}.  As shown in Figure~\ref{Kqtbandrproof}, positive or negative dots on the boundary of the box-dot template correspond to downward- or upward-oriented cusps in the front projection of $K_q$, respectively.  Since the only cusps in $K_q$ occur at the $2Q$ boundary dots of the corresponding template, by summing over either side of the template we deduce that 
\begin{align*}
D &= \left\{\begin{array}{l l} 
Q/2+Q/2 = Q & \textrm{if } Q \textrm{ even, } P \textrm{ odd}\\ 
(Q+1)/2+(Q+1)/2 = Q+1 & \textrm{if } Q \textrm{ odd, } P \textrm{ even}\\
(Q+1)/2+(Q-1)/2 = Q & \textrm{if } Q \textrm{ odd, } P \textrm{ odd}\\\end{array}\right.\\
U &= \left\{\begin{array}{l l}
Q/2+Q/2 = Q & \textrm{if } Q \textrm{ even, } P \textrm{ odd}\\
(Q-1)/2+(Q-1)/2 = Q-1 & \textrm{if } Q \textrm{ odd, } P \textrm{ even}\\
(Q-1)/2+(Q+1)/2 = Q & \textrm{if } Q \textrm{ odd, } P \textrm{ odd}\\\end{array}\right.
\end{align*}

\begin{figure}[ht]
\begin{center}
\includegraphics[width=0.5\textwidth]{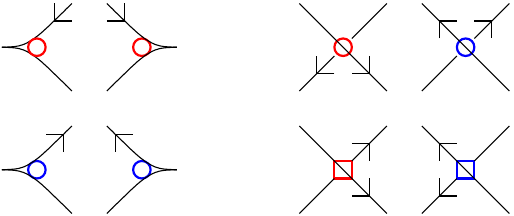}
\caption[Boxes and dots, crossings and cusps.]{At left: positive boundary dots give downward cusps, negative boundary dots give upward cusps.  At right: interior dots give negative crossings, interior boxes give positive crossings.}
\label{Kqtbandrproof}
\end{center}
\end{figure}

Further, the nature of crossings of $K_q$ implies that those crossings occurring at interior boxes in the box-dot template are positive (right-handed), while those at interior dots are negative (left-handed).  Thus the writhe $wr(K_q)$ of $K_q$ can be computed by subtracting the number of interior dots in the template from the number of interior boxes.  As there are a total of $P(Q-1)$ interior boxes and $Q(P-1)$ interior dots in the box-dot template, we have that
\begin{align*}
tb(K_q) &= wr(K_q) - \frac{1}{2}(D+U)\\
&= P(Q-1)-Q(P-1) - Q\\
&= -P,\\
r(K_q) &= \frac{1}{2}(D-U)\\
&= \left\{\begin{array}{l l} 0 & \textrm{if } P \textrm{ is odd,}\\ 1 & \textrm{if } P \textrm{ is even.}\end{array}\right.\qedhere
\end{align*}
\end{proof}

\section{Encoding Regular Tangles in Characteristic Foliations}
\label{sec:tanglesandfoliations}
We saw in Section~\ref{sec:boxdotconstructions} that a box-dot diagram $\boxdot_{q^f}$ contains enough information to recreate both a regular Legendrian rational tangle $\G_{q^f}$ and the associated compressing disc $D_{q^f}$ bearing a particular signed characteristic foliation.  In Section~\ref{sec:results}, our goal will be to use box-dot diagrams as a sort of combinatorial invariant for distinguishing Legendrian tangles.  The means by which we will accomplish this lie in Lemmas~\ref{cardinalitylemma} and~\ref{bijectionlemma}, found in Section~\ref{sec:lemmas}.  These lemmas establish necessary conditions on the corresponding foliations for the existence of an isotopy between two regular tangles which can be read directly from their box-dot diagrams.  The proofs of these lemmas are somewhat technical, and will require additional modifications of the corresponding compressing discs.  These modifications, described in Section~\ref{sec:bubble}, will produce a surface that still depends only on our earlier constructions, but whose characteristic foliation admits an embedding, constructed in Section~\ref{sec:push}, of the strands of $\G_{q^f}$.  The motivation underlying this section is the suspicion that if there exists a Legendrian isotopy between curves which are themselves leaves of the characteristic foliation of some surface, then there ought to be a deformation of the surface---realizable as an isotopy of its foliation---which realizes this Legendrian isotopy.  Lemmas~\ref{cardinalitylemma} and~\ref{bijectionlemma} represent a less restrictive observation than this, though one that suffices in proving the results of Section~\ref{sec:results}.

In sections~\ref{sec:push} and~\ref{sec:lemmas} we will define a collection of \emph{ambient contact isotopies}---isotopies of $(B_q,\xi)$ through \emph{contactomorphisms}, or diffeomorphisms which also preserve the contact structure.  We remark here that we can, and will, assume that any such ambient contact isotopy fixes the boundary of $B_q$; particularly the unknot $K_q$ and the tangle endpoints $p_1$, $p_2$, $p_3$, and $p_4$.

\subsection{Bubbling and Warping}
\label{sec:bubble}
We assume first that $\G_{q^f}$ is an odd-length tangle, i.e., that $q\sim(q_n,\ldots,q_1)$ with $n$ odd.  The subdivision process described in section~\ref{sec:boxdotdiagrams} can be viewed as adding a collection of line segments to the rectangle $\R_q$ so as to produce the squares  of the subdivision.  Thus the terminal stage of the construction of $\boxdot_{q^f}$ can be seen as subdividing the remainder rectangle $\r_n$ into a row of $q_n$ squares of unit area by adding $q_n-1$ unit-length vertical line segments spaced evenly along $\r_n$.  These segments in turn provide us with a set of $q_n-1$ dots in the box-dot diagram which we will refer to as \emph{shared dots}.  Similarly, all of the remaining stages of the construction of the subdivision can be described as adding a collection of $q_j$ segments---vertical if $j$ is odd, horizontal if $j$ is even---to the remainder rectangle $\r_j$, for $1\leq j\leq n-1$.  Respectively, from each row or column of boxes or dots that are associated to these segments, we isolate the left-most and right-most boxes or top-most and bottom-most dots, referring to these collectively with the shared dots as \emph{tagged boxes} or \emph{tagged dots}.  The tagged boxes of $\boxdot_{q^f}$ are in one-to-one correspondence with the cusps of the front projection of $\G_{q^f}$, while the tagged dots correspond\footnote{This correspondence is one-to-one except at shared elliptics.} to local extrema of its strands.  Finally, we isolate the top-most and bottom-most of the boundary dots from either side of the template, referring to these as \emph{endpoint dots}.  These various subsets of the box-dot diagram are illustrated in Figure~\ref{boxdottypes} to provide the intuition behind the terminology we employ.  

In constructing the compressing disc $D_{q^f}$ from $\boxdot_{q^f}$, we produced a correspondence between the elliptic points of the characteristic foliation and the dots---both interior dots and boundary dots---of $\boxdot_{q^f}$ (cf.\ Figure~\ref{bddfoliation}).  In what follows we will refer to subsets of the elliptic points of $D_{q^f}$ as \emph{shared elliptics, tagged elliptics, or endpoint elliptics}, associating to them the same descriptors as may have been endowed upon the dots that they arose from.

\begin{figure}[ht]
\begin{center}
\includegraphics[width=0.8\textwidth]{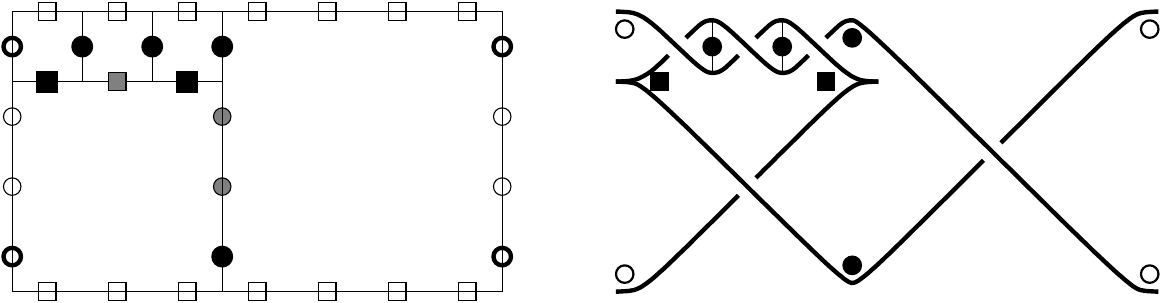}
\caption[Shared dots, tagged dots/boxes, and endpoint dots.]{The shared dots, tagged dots and boxes, and endpoint dots of $\boxdot_{7/4}$ are shown in bold at left.  At right we show only these points with $\G_{7/4}$ to justify our teminology---tagged points are nearest the front projection of the tangle, with shared dots equidistant from both strands, while endpoint dots are nearest the ends of the tangle.}
\label{boxdottypes}
\end{center}
\end{figure}

The first modification we perform will convert $D_{q^f}$ into a particular \emph{branched surface}\footnote{We avoid defining branched surfaces here, as we have no need for much of the technology associated to them.  While this term is used as a frame of reference for those familiar with such objects, the properties of the surface we employ should be clear from the exposition.}.  Specifically, for some small $\epsilon>0$, at each shared elliptic $e$ in the characteristic foliation of $D_{q^f}$ we replace the closed neighborhood $\{p \in D_{q^f}\big| |p-e| \leq \epsilon\}$ with the sphere $S_{\epsilon}(e)=\{p \in \rr^3 \big| |p-e| = \epsilon \}$.  Viewing this sphere as a pair of discs which meet along the \emph{equator} curve $C_{\epsilon}(e)=\{ p \in D_{q^f} \big| |p-e| = \epsilon \}$, we then smooth each of these discs in such a way that they meet $D_{q^f}$ with a well-defined tangent plane along this equator (see Figure~\ref{bubbling}).  This process produces what resembles a bubble in the disc $D_{q^f}$ around each shared elliptic $e$.  Thus, we will refer to this modification as \emph{bubbling} $D_{q^f}$, and will refer to the resulting object as a \emph{bubbled disc}.  We may assume that the characteristic foliation of the bubbled disc, where it does not already agree with that of $D_{q^f}$, is obtained from the standard foliation on the sphere $S_{\epsilon}(e)$ (cf.\ Figure~\ref{folball}) by ``pinching'' this sphere along $C_{\epsilon}(e)$.  In particular, each bubble contains only two elliptic singularities---which we shall refer to as \emph{poles}---at the endpoints of the vertical segment of length $2\epsilon$ centered at $e$.  There are natural isotopies which carry the disc $D_{q^f}$ to either branch while preserving the foliation, up to diffeomorphism, of the affected portion.  These isotopies induce a natural choice of co-orientation on the corresponding bubbled disc, associating signs to both poles of any bubble that agree with the sign of the corresponding shared elliptic.  We may also assume the equator $C_{\epsilon}(e)$ of a bubble will intersect the characteristic foliation of the bubbled disc near $e$ transversely, a fact we will require in the proof of Lemma~\ref{bijectionlemma}.

\begin{figure}[ht]
\begin{center}
\includegraphics[width=.8\textwidth]{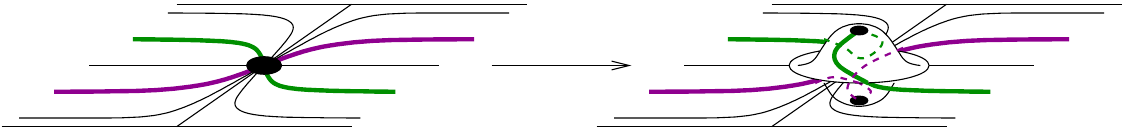}
\caption[Bubbling $D_{q^f}$ near a shared elliptic point.]{Bubbling $D_{q^f}$ near a shared elliptic point.  Note that bubbling allows leaves in the foliation which would have intersected at the shared elliptic to pass over one another without intersecting.}
\label{bubbling}
\end{center}
\end{figure}  

For our second modification, at each hyperbolic-hyperbolic connection in the foliation of $D_{q^f}$ corresponding to a tagged box, we push the midpoint of this connection slightly in either the positive or negative $x$-direction toward the nearest cusp of $\G_{q^f}$.  This operation can be realized as an ambient isotopy of $B_q$, and will result in a shearing of the highly unstable hyperbolic-hyperbolic connection (see Figure~\ref{hyperbolicshearing}).  This shearing creates an opening in the foliation of $D_{q^f}$ which allows leaves of the foliation to pass from one of the trivially-foliated subdiscs described in Section~\ref{sec:compressingdisc} to a vertically adjacent one.  We will refer to this modification as \emph{warping} $D_{q^f}$, and will refer to the resulting object as a \emph{warped disc}.  

We remark here that if this operation results in the shearing of several of these connections in succession then the ends of these passages can be aligned, producing leaves in the foliation which traverse all of these passages without encountering an elliptic point.  Such sequential shearings  may introduce new singularities into the characteristic foliation of $D_{q^f}$, but these singularities can be isolated from the leaves of the foliation which traverse the resulting passages.  In what follows we will only be concerned with the foliations near such leaves, and so omit the precise nature of this alteration here, referring the reader instead to the appendices of~\cite{S11} for a more thorough treatment.

\begin{figure}[ht]
\begin{center}
\includegraphics[width=\textwidth]{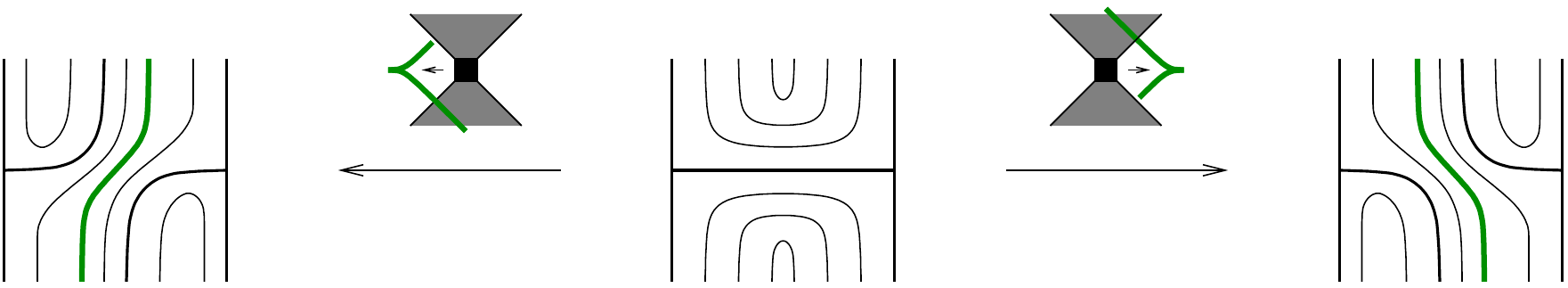}
\caption[Warping $D_{q^f}$ near hyperbolic-hyperbolic connections.]{The two forms of shearing that can occur when warping the disc $D_{q^f}$ near hyperbolic-hyperbolic connections.  The hyperbolic point on the left in each figure lies on the overstrand at the indicated box in the box-dot diagram.}
\label{hyperbolicshearing}
\end{center}
\end{figure}

Applying both of these modifications to $D_{q^f}$ then produces a \emph{warped, bubbled disc} $\wbD_{q^f}$ which agrees with the foliated disc $D_{q^f}$ except on small neighborhoods of either shared elliptics or centers of hyperbolic-hyperbolic connections, or on larger portions of the disc near sequential hyperbolic-hyperbolic shearings where this difference will be negligible in later applications.

Before continuing on, we remark that even-length tangles will prove to be resistant to our techniques throughout the remainder of this work.  It is in the application of the above operations that this disparity begins to become apparent.  If we begin with an even-length tangle, there will be no shared elliptics in the box-dot diagram, as the remainder rectangle $\r_n$ instead yields $q_n-1$ \emph{shared boxes}.  As such, we are left without the ability to shear the corresponding hyperbolic-hyperbolic connections in both directions simultaneously in a way that would allow us to record this modification in the foliation.  One approach to resolving this that will bear some merit in later work is to instead produce a pair of foliated discs, denoted $\wbD_{q^f}^1$ and $\wbD_{q^f}^2$, obtained by applying only the warping operation separately to $D_{q^f}$ relative to the individual strands\footnote{Our labeling convention on these strands will be defined formally in Section~\ref{sec:lemmas}.} $\g_{q}^1$ and $\g_{q}^2$ of $\G_{q^f}$, respectively.  Each of these discs will again agree with the foliation of the original compressing disc $D_{q^f}$ where relevant, but the information lost by not being able to identify these discs with each other near shared boxes in the diagram will limit their use considerably in Section~\ref{sec:results}. 

\subsection{Pushing and Pulling}
\label{sec:push}
Given an odd-length regular Legendrian rational tangle $\G_{q^f}$, we consider the associated box-dot diagram $\boxdot_{q^f}$ and the corresponding foliated compressing disc $D_{q^f}$.  The operations we describe in this section will provide us with two different methods of embedding the tangle $\G_{q^f}$, up to neighborhoods of its endpoints, in the characteristic foliation of the warped, bubbled disc $\wbD_{q^f}$

The first of these operations, \emph{pulling}, consists of a Legendrian isotopy of tangles $\{\psi_t\}_{t \in [0,1]}$ which we apply to the strands of $\G_{q^f}$.  Intuitively, this isotopy will serve to tighten the tangle, drawing it in toward $\wbD_{q^f}$. More precisely, this operation can be defined by way of a transverse isotopy of the front projection of $\G_{q^f}$, illustrated in Figure~\ref{pulloperation}.  

\begin{figure}[ht]
\begin{center}
\includegraphics[width=\textwidth]{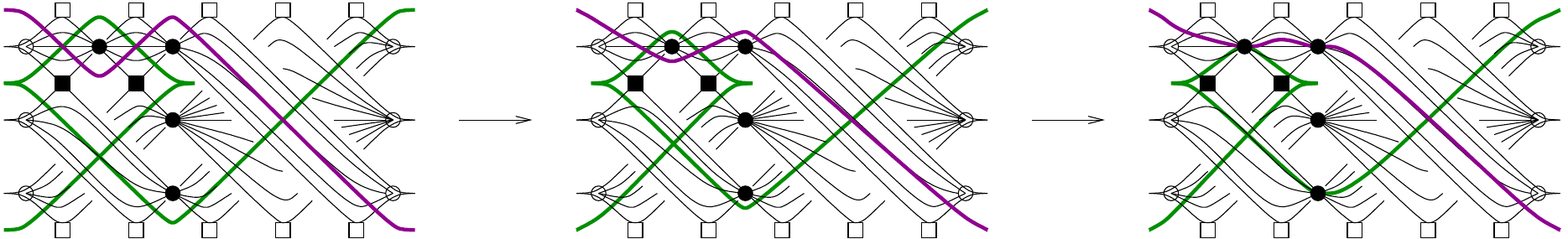}
\caption[The pull operation.]{The transverse isotopy of the front projection of $\G_{5/3}$ which realizes the corresponding pull operation.  To illustrate the homotopy of the entire tangle we include the portions of the tangle with positive slope, which should lie underneath the compressing disc shown.}
\label{pulloperation}
\end{center}
\end{figure}

Since this homotopy of fronts translates to a Legendrian isotopy, we have that the resulting tangle still consists of a pair of Legendrian curves.  More importantly, we can in fact pull the tangle until it intersects $\wbD_{q^f}$ along all but a neighborhood of the endpoints $\{p_1,p_2,p_3,p_4\}\in \partial B_q$.  The positioning of elliptic points in the foliation of $D_{q^f}$ allow for large portions of the tangle to be embedded in the foliation without any adaptation.  However, to be able to perform the strandwise isotopies of the pull operation simultaneously, we must avoid the inherent intersection of the strands that would occur at the shared elliptics.  By bubbling the disc, we prevent such intersections by allowing the individual strands to pass above and below the original disc $D_{q^f}$.  By warping the disc we produce leaves in the foliation which are Legendrian isotopic to the portions of the strands which pass through cusps (near tagged boxes) in their front projection.  This allows these portions of the tangle to be embedded in the characteristic foliation of the warped disc where they would have cut across hyperbolic-hyperbolic connections in $D_{q^f}$.  

As $\wbD_{q^f}$ remains fixed during the pull operation, it must already contain the pair of leaves that will become the image of $\G_{q^f}$.  We will denote by $\PG_{q^f}$ the Legendrian tangle consisting of these leaves, away from the endpoint elliptics, along with Legendrian arcs in $B_q$ which join these leaves to the endpoints $\{p_1,p_2,p_3,p_4\}$.  To clarify our persistent comments about endpoints, we remark here that because the boundary of the ambient ball $B_q$ must remain fixed, we cannot force the endpoints of $\G_{q^f}$ to intersect the bounding unknot $K_{q^f}$ during the pull operation, preventing us from fully embedding the tangle as a leaf in the foliation.  To account for this, we fix small neighborhoods of the boundary elliptics of $D_{q^f}$ nearest these endpoints (see Figure~\ref{egflatdisc}), requiring that the strands of $\PG_{q^f}$ be embedded as leaves in the foliation of $\wbD_{q^f}$ until they intersect the boundaries of these neighborhoods, at which point the strands leave the disc and travel directly to the corresponding endpoints on $\partial B_q$. 

\begin{figure}[ht]
\begin{center}
\includegraphics[width=0.4\textwidth]{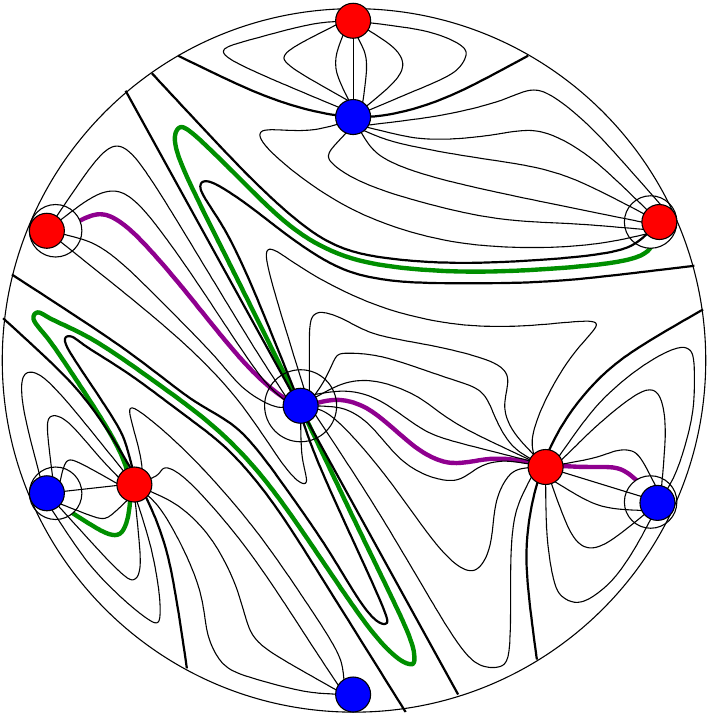}
\caption[A planar diagram of the disc $\wbD_{5/3}$.]{The compressing disc $D_{5/3}$ (cf.\ Figure~\ref{egbddisc}), under an identification with the unit disc in $\rr^2$.  To make this coherent with our previous diagrams, we assume a counterclockwise orientation on the boundary of this disc, with the endpoint elliptic in the top-left corner of the box-dot diagram corresponding to the top-left positive boundary elliptic in this figure.}
\label{egflatdisc}
\end{center}
\end{figure}

It is a classical result of contact geometry (Theorem 2.6.2 in~\cite{G08}) that given a Legendrian isotopy $\{\psi_t\}_{t \in [0,1]}$ of curves, there is a corresponding \emph{ambient contact isotopy} $\{\Psi_t\}_{t \in [0,1]}$ of $B_q$ which realizes this Legendrian isotopy.  The \emph{push operation} on the warped, bubbled disc $\wbD_{q^f}$ is then simply the effect on $\wbD_{q^f}$ of the ambient contact isotopy $\{\Psi_t\}_{t \in [0,1]}$ that this theorem associates to the reverse $\{\psi_{1-t}\}_{t \in [0,1]}$ of the pull operation $\{\psi_t\}$.  We will denote by $\pD_{q^f}$ the image $\Psi_1(\wbD_{q^f})$ of $\wbD_{q^f}$ under the push operation.  Because of the ambient nature of this operation, we will not only move the tangle $\PG_{q^f}$ inside $B_q$, but will also move the disc $\wbD_{q^f}$ along with it.  We can assume that this ambient contact isotopy is trivial over $\wbD_{q^f}$ except on some small neighborhood of $\PG_{q^f}$.  Since the push operation is an ambient contact isotopy, the characteristic foliation of $\wbD_{q^f}$ must be preserved, up to diffeomorphism, throughout this operation.    

To summarize the operations and notation developed in this and the preceding section, given a box-dot diagram $\boxdot_{q^f}$ for an odd-length Legendrian rational tangle $\G_{q^f}$ we first construct the compressing disc $D_{q^f}$ as in Section~\ref{sec:compressingdisc}.  We then modify specific portions of $D_{q^f}$ to produce the warped, bubbled disc $\wbD_{q^f}$ which contains in its foliation a tangle $\PG_{q^f}$.  The pull operation provides an isotopy from $\G_{q^f}$ to $\PG_{q^f}$, up to endpoint considerations, while the push operation pushes the disc $\wbD_{q^f}$ out to a disc $\pD_{q^f}$ carrying $\PG_{q^f}$ to $\G_{q^f}$.  

If instead $\G_{q^f}$ is an even-length tangle, we may still apply the push and pull operations as described above, but only to each of the discs $\wbD_{q^f}^{\ell}$ for $\ell\in\{1,2\}$ individually.  In this way we obtain a Legendrian arc in each of these discs which is the image of one of the strands $\g_{q}^1$ and $\g_{q}^2$ under the respective pull operations.  Similarly, for each $\ell$ we obtain a disc $\pD_{q^f}^{\ell}$, the image of $\wbD_{q^f}^{\ell}$ under the respective push operations, which contains one of the individual strands of $\G_{q^f}$.

\subsection{Necessary Conditions for Legendrian Isotopy}
\label{sec:lemmas}
Lemmas~\ref{cardinalitylemma} and~\ref{bijectionlemma} below represent the operative core of the work in this paper, as they will be our primary tools in distinguishing Legendrian flypes of regular Legendrian rational tangles in Section~\ref{sec:results}.  Before stating these lemmas, however, we must establish some additional notation referring to particular sets of elliptic points in the various discs we've constructed.  

As a matter of convention we will label the individual strands of any tangle $\G_q$ as $\g_q^1$ and $\g_q^2$ such that $\g_q^1$ will refer to the strand of $\G_q$ containing the marked point $p_1\in\partial B_q$ as one of its endpoints.  We orient the strand $\g_q^1$ so that it is traversed from $p_1$ to its other endpoint.  We then orient the other strand $\g_q^2$ in such a way that both strands are traversed in the same direction\footnote{Either both to the left or both to the right in the front projection if $q>1$, or both up or both down if $q<1$.} across the remainder rectangle $\r_n$.  These conventions provide a consistent choice of strand labeling and orientation over all the various $f$-moves on the box-dot diagram described in Section~\ref{sec:fmoves}, since these all correspond to topological isotopies of the given tangle which fix the endpoints of the strands.  While we will rely on this orientation frequently for combinatorial purposes, any of the arguments which employ our choice of orientation can be adapted easily to accommodate any other choice of orientation on the strands of $\G_q$.

To simplify the following discussion, we describe our notation explicitly only for odd-length tangles.  The same notation can be used for even-length tangles as well by replacing any mention of discs $\wbD$ and $\pD$ with the respective pairs of discs $\wbD^\ell$ and $\pD^\ell$, $\ell\in\{1,2\}$, discussed in sections~\ref{sec:bubble} and~\ref{sec:push}.  In particular, Lemmas~\ref{cardinalitylemma} and~\ref{bijectionlemma} apply to tangles of either even or odd length.

For any odd-length regular Legendrian rational tangle $\G_{q^f}$, we consider the foliated compressing disc $D_{q^f}$ arising from the box-dot diagram $\boxdot_{q^f}$.  We will denote by $E_{q^f}$ the set of tagged elliptics in $D_{q^f}$.  Similarly, we will denote by $\wbE_{q^f}$ and $\pE_{q^f}$ the sets of elliptic points which intersect the strands of $\PG_{q^f}$ in the warped, bubbled disc $\wbD_{q^f}$ and $\G_{q^f}$ in $\pD_{q^f}$ after the push operation, respectively.  By construction, $E_{q^f}$ and $\wbE_{q^f}$ are nearly identical, differing only between shared elliptics and poles of bubbles.  Thus there is a well-defined map from $\wbE_{q^f}$ to $E_{q^f}$ obtained by collapsing each bubble, identifying its poles with the corresponding shared elliptic.

Following our labeling convention on the strands of $\G_{q^f}$, we separate each of the sets $\wbE_{q^f}$ and $\pE_{q^f}$ into disjoint subsets $\wbE_{q^f}=\wbE_{q^f}^1\sqcup\wbE_{q^f}^2$ and $\pE_{q^f}=\pE_{q^f}^1\sqcup\pE_{q^f}^2$ corresponding to the elliptic points in the respective discs which intersect the individual strands $\g_q^1$ and $\g_q^2$ of $\G_{q^f}$.  Under the orientation on the strands of $\G_{q^f}$ given above, we induce an ordering on these subsets in which one elliptic point is deemed less than another if it occurs earlier along the corresponding strand of $\G_{q^f}$ or $\PG_{q^f}$, accordingly.  Under the map from $\wbE_{q^f}$ to $E_{q^f}$ we induce a similar decomposition of $E_{q^f}=E_{q^f}^1\cup E_{q^f}^2$, where $E_{q^f}^1\cap E_{q^f}^2$ is precisely the set of shared elliptics of $D_{q^f}$.  This map induces a similar ordering on the subsets $E_{q^f}^1$ and $E_{q^f}^2$.  We remark here that our prescribed orientation then produces a consistent ordering on the set of shared elliptics in $D_{q^f}$, though this will not be required in our later work.    

For two regular Legendrian rational tangles $\G_{q^f}$ and $\G_{q^g}$, it is easily seen that the cardinalities of the sets $E_{q^f}$ and $E_{q^g}$ are equal by counting the numbers of tagged dots in the corresponding box-dot diagrams.  The cardinalities of $\wbE_{q^f}$ and $\wbE_{q^g}$ are therefore also equal, as they differ from the previous counts only by doubling the number of shared elliptics.  Since contactomorphisms preserve characteristic foliations,  the cardinalities of $\pE_{q^f}$ and $\pE_{q^g}$ are equal as well.  The cardinalities of the strandwise versions of these sets, however, are more difficult to compute, as they depend on the individual parities of the twist components $q_j$.

\begin{cardinalitylemma}
\label{cardinalitylemma}
Let $\G_{q^f}$ and $\G_{q^g}$ be regular Legendrian rational tangles, and suppose there is a Legendrian isotopy from $\G_{q^f}$ to $\G_{q^g}$.  For each $\ell\in\{1,2\}$, the cardinalities of $E_{q^f}^{\ell}$, $E_{q^g}^{\ell}$, $\wbE_{q^f}^{\ell}$, $\wbE_{q^g}^{\ell}$, $\pE_{q^f}^{\ell}$ and $\pE^\ell_{q^g}$ are all equal.
\end{cardinalitylemma}

\begin{proof}
It suffices to prove for each $\ell$ that the cardinalities of $\pE_{q^f}^{\ell}$ and $\pE^\ell_{q^g}$ are equal, as the contactomorphisms given by the inverses of the push operations for the respective discs, restricted to the corresponding strand of the respective tangles, will then extend the proof to $\wbE_{q^f}^{\ell}$ and $\wbE_{q^g}^{\ell}$.  Collapsing the bubbles in $\wbD_{q^f}$ and $\wbD_{q^g}$ will further extend the proof to $E_{q^f}^{\ell}$ and $E_{q^g}^{\ell}$, establishing the lemma.

As mentioned in Section~\ref{sec:push}, the existence of a Legendrian isotopy $\{\varphi_t\}_{t \in [0,1]}$ from $\G_{q^f}$ to $\G_{q^g}$ implies the existence of an ambient contact isotopy $\{\Phi_t\}_{t \in [0,1]}$ of $B_q$ which realizes this Legendrian isotopy.  The intersection of the image $\Phi_1(\pD_{q^f})$ with $\pD_{q^g}$ contains the Legendrian tangle $\G_{q^g}=\Phi_1(\G_{q^f})$ as a leaf in its characteristic foliation, up to the usual endpoint considerations.  For the duration of the proof, we will modify our convention regarding endpoints to add rigidity to the argument below.  Specifically, we instead focus on the entirety of the leaves in the respective characteristic foliations corresponding to each of the tangles.  Thus, instead of leaving the discs to travel to the marked endpoints on $B_q$, the embedded tangles will now end at the endpoint elliptics in the respective discs.

To prove the lemma, we first consider a map from a neighborhood in $\mathbb{R}^3$ of the first strand of $\G_{q^g}$ to the cylinder $\{(u,v)\in\rr^2|u^2+v^2\leq 1\}\times I$ where $I$ is a closed interval.  Specifically, we map the portion of $\pD_{q^g}$ inside this neighborhood to $\{(u,0)|-1\leq u\leq 1\} \times I$, such that $\g_q^1$ is identified with $\{(0,0)\}\times I$.  Under this identification, the unit normal vectors to the contact planes along $\g_q^1$ trace out what we will call a \emph{contact curve} on the boundary of this cylinder.  Similarly, the unit normal vectors to $\pD_{q^g}$ and their negatives along $\g_q^1$ produce a pair of curves that we will call \emph{normal curves} of $\pD_{q^g}$, given by $\{(0,\pm 1)\}\times I$.  The singularities of the characteristic foliation of $\pD_{q^g}$ then occur precisely when the contact curve intersects the normal curves on this cylinder.  The type of each singularity is determined by the direction in which the contact curve crosses the normal curves, as in Figure~\ref{foliationtube}.  Thus, in general, the number of elliptic points minus the number of hyperbolic points along $\g_q^1$ gives a measure of how much the contact structure twists relative to the disc along this strand.  

Using a similar identification, we can obtain a measure of the net twisting of $\Phi_1(\pD_{q^f})$ along $\g_q^1$.  Since the discs $\Phi_1(\pD_{q^f})$ and $\pD_{q^g}$ share a common boundary in the unknot $K_q$, the net twisting of these surfaces about $\g_q^1$ must be the same.  Thus, if these discs contain different numbers of elliptic singularities along this strand, then one of these surfaces must also contain hyperbolic points along this strand.  Since such hyperbolic points do not occur in our constructions, and cannot be introduced along the strand through our modifications or contactomorphisms, we have that both discs $\Phi_1(\pD_{q^f})$ and $\pD_{q^g}$ must contain the same number of elliptic singularities, and no hyperbolic singularities, along $\g_q^1$.  Since $\Phi_1$ preserves the characteristic foliation of $\pD_{q^f}$, and this count of elliptic singularities only includes the two additional boundary elliptics that form the modified endpoints of the strand, the cardinalities of $\pE_{q^f}^1$ and $\pE_{q^g}^1$ must be equal.  An identical argument applied to the other strand $\g_q^2$ then establishes the equality of the cardinalities of $\pE_{q^f}^2$ and $\pE_{q^g}^2$.
\end{proof}

\begin{figure}[ht]
\begin{center}
\includegraphics[width=.6\textwidth]{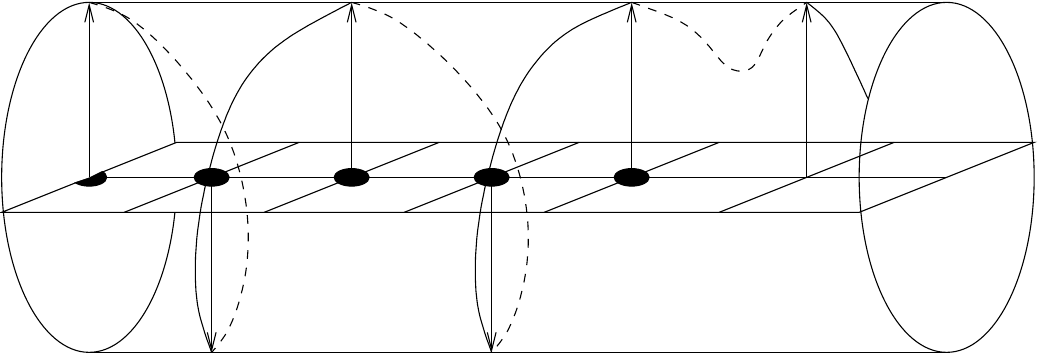}
\caption[Singularities measure the twisting of $\xi$ relative to a surface.]{The net twisting of the contact structure along an arc in a surface can be measured by counting elliptic and hyperbolic singularities along the arc.}
\label{foliationtube}
\end{center}
\end{figure}

\begin{bijectionlemma}
\label{bijectionlemma}
Let $\G_{q^f}$ and $\G_{q^g}$ be regular Legendrian rational tangles, and suppose there is a Legendrian isotopy from $\G_{q^f}$ to $\G_{q^g}$.  Then there exists a bijection $\beta$ from $E_{q^f}$ to $E_{q^g}$ which restricts to order-preserving bijections $\beta^{\ell}$ from $E_{q^f}^{\ell}$ to $E_{q^g}^{\ell}$, for each $\ell \in \{1,2\}$.
\end{bijectionlemma}

\begin{proof}
We assume first that $\G_{q}$ is an odd-length tangle.  To prove the lemma in this case, it will suffice to construct a bijection $\widetilde{\beta}$ from $\wbE_{q^f}$ to $\wbE_{q^g}$ which restricts to order-preserving bijections from $\wbE_{q^f}^{\ell}$ to $\wbE_{q^g}^{\ell}$, for each $\ell \in \{1,2\}$, and maps the poles of each bubble in $\wbD_{q^f}$ to the poles of a bubble in $\wbD_{q^g}$.  The map $\beta$ will then be defined to agree with $\widetilde{\beta}$ on all but the shared elliptics, where $\beta$ will instead be obtained from $\widetilde{\beta}$ via the natural identifications between these points and the poles of their corresponding bubbles.  

To construct the map $\widetilde{\beta}$, we will require an additional modification---an ambient contact isotopy that we will refer to as the \emph{press operation}---to our existing constructions.  As in the proof of Lemma~\ref{cardinalitylemma}, there is an ambient contact isotopy $\{\Phi_t\}_{t\in[0,1]}$ of $B_q$ which realizes the Legendrian isotopy from $\G_{q^f}$ to $\G_{q^g}$.  Recall too that the image $\Phi_1(\pD_{q^f})$ and the disc $\pD_{q^g}$ intersect along the tangle $\G_{q^g}$.  By our constructions, the tangle in each of these discs consists only of elliptic singularities and leaves of the foliation which connect them smoothly.  Further, Lemma~\ref{cardinalitylemma} ensures that we have the same number of elliptic singularities in either disc along each strand of $\Gamma_{q^g}$.  Small neighborhoods of this tangle in both discs will then contain diffeomorphic foliations, differing only in the exact locations of these elliptic singularities along the individual strands.  Thus, by repositioning these singularities, we can construct yet another disc bounded by $K_q$ whose characteristic foliation agrees with $\Phi_1(D_{q^f})$ away from $\G_{q^g}$, and which agrees with $\pD_{q^g}$ near $\G_{q^g}$.  This disc is obtained intuitively by pressing $\Phi_1(D_{q^f})$ onto $\pD_{q^g}$ along a small neighborhood of the tangle.  Since $\Phi_1(D_{q^f})$ and this new disc have diffeomorphic foliations and the same boundary, a theorem of Giroux\footnote{See Theorem 2.5.22 in~\cite{G08}.  However, this theorem, as stated, applies specifically to closed surfaces.  To rectify this, we can glue a disc in $\rr^3-B_q$ to $\Phi_1(\pD_{q^f})$ along $K_q$, producing a sphere.  Since this (now closed) surface bears diffeomorphic foliations before and after the press modification, we may appeal to the proof of the cited theorem.} gives that this operation can in fact be performed by way of an ambient contact isotopy $\{\Upsilon_t\}_{t\in[0,1]}$.

To prove the lemma, we will also employ the push operations $\{\Psi_t^f\}$ and $\{\Psi_t^g\}$ which carry $(\wbD_{q^f},\PG_{q^f})$ to $(\pD_{q^f},\G_{q^f})$ and $(\wbD_{q^g},\PG_{q^g})$ to $(\pD_{q^g},\G_{q^g})$, respectively.  The smooth concatenation of the isotopies $\{\Psi_t^f\}$, $\{\Phi_t\}$, $\{\Upsilon_t\}$, and $\{\Psi_{1-t}^g\}$ yields an ambient contact isotopy $\{\Omega_t\}_{t\in[0,1]}$ of $B_q$ whose existence incorporates all of the technology developed in sections~\ref{sec:boxdotdiagrams}, \ref{sec:boxdotconstructions}, and~\ref{sec:tanglesandfoliations}.  In particular, the isotopy $\{\Omega_t\}$ first carries $(\wbD_{q^f}, \PG_{q^f})$ to $(\pD_{q^f}, \G_{q^f})$ through contactomorphisms, preserving the characteristic foliation and giving us a natural identification between $\wbE_{q^f}$ and $\pE_{q^f}$.  Similarly, we obtain identifications between $\pE_{q^f}$ and $\Phi_1(\pE_{q^f})$, and $\Phi_1(\pE_{q^f})$ and $\Upsilon_1\circ\Phi_1(\pE_{q^f})=\pE_{q^g}$ from the portions of $\{\Omega_t\}$ corresponding to the given Legendrian isotopy and the press operation, respectively.  Finally, the last stage of $\{\Omega_t\}$ pulls the pressed, carried, pushed, warped, bubbled disc $\Upsilon_1\circ\Phi_1(\pD_{q^f})$ down to intersect with $\wbD_{q^g}$ along a neighborhood of $\PG_{q^g}$, giving an identification between $\pE_{q^g}$ and $\wbE_{q^g}$.  Furthermore each of these identifications between sets of elliptic singularities must also preserve the ordering of these sets induced by the strandwise orientations, as the endpoints of the tangle are fixed throughout.  Thus we define the map $\widetilde{\beta}$ to be the restriction of $\Omega_1$ to $\wbE_{q^f}$.

To complete the proof we need only verify that the bijection $\widetilde{\beta}$ is respective of the bubbles.  To this end, we recall from Section~\ref{sec:bubble} that for each bubble in $\wbD_{q^f}$ arising from a shared elliptic $e$ in $D_{q^f}$, the equator $C_{\epsilon}(e)$ of the bubble consists of a single simple closed curve which is transverse to the foliation of $\wbD_{q^f}$.  By construction, the poles of this bubble are the only singularities of the characteristic foliation of $\wbD_{q^f}$ inside this curve.  Since ambient contact isotopies preserve the transversality of this curve, its image $\Upsilon_1\circ\Phi_1\circ\Psi_1^f(C_{\epsilon}(e))$ must still enclose exactly two singularities---the images of these poles.  If we suppose to the contrary that these singularities are not themselves the poles of a single bubble in $\pD_{q^g}$ then instead they must correspond either to distinct tagged, but not shared, elliptics, to poles of two distinct bubbles, or to one tagged, but not shared, elliptic and one pole of a bubble.  In any of these cases, some portion of the characteristic foliation of the image $\Upsilon_1\circ\Phi_1(\pD_{q^f})$ must be isotopic to the region depicted in Figure~\ref{nontransverse}.  This figure also serves to illustrate that such a closed curve cannot enclose exactly two elliptic singularities and still be transverse to the surrounding foliation, providing a contradiction.  Thus the poles of each bubble in $\wbD_{q^f}$ must correspond to the poles of a single bubble in $\wbD_{q^g}$, and the map $\widetilde{\beta}$ provides the desired bijection.

If instead we assume that $\G_q$ is an even-length tangle, the proof is similar, though somewhat simpler.  We consider the discs $\wbD_{q^f}^1$ and $\wbD_{q^f}^2$ independently, and construct a corresponding ambient contact isotopy $\{\Omega_t^{\ell}\}$ for each $\ell\in\{1,2\}$.  We define $\widetilde{\beta}^{\ell}$ as the restriction of $\Omega_1^{\ell}$ to $\wbE_{q^f}^{\ell}$.  Defining $\widetilde{\beta}$ to be the union of these two maps then completes the proof, as there are no shared elliptics, and hence no bubbles, to be concerned with in this case.
\end{proof}

\begin{figure}[ht]
\begin{center}
\includegraphics[width=0.25\textwidth]{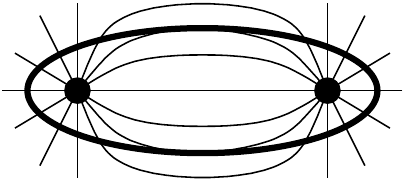}
\caption[A non-transverse curve.]{A simple closed curve which encloses exactly two elliptic singularities of a characteristic foliation cannot be transverse to that foliation.}
\label{nontransverse}
\end{center}
\end{figure}

\section{On the Classification of Legendrian Flypes} 
\label{sec:results}
Here we prove a collection of results addressing the question of whether Legendrian flypes of regular Legendrian rational tangles can be realized via Legendrian isotopy\footnote{For simplicity, we will hereafter refer to Legendrian flypes, regular Legendrian rational tangles, and Legendrian isotopies as flypes, tangles, and isotopies, respectively, except in the formal statements of theorems.}.  The first of these results, Proposition~\ref{trivialflypes}, is well-known.  We include this statement for completeness, though we omit its trivial proof.

\begin{trivialflypes}
\label{trivialflypes}
The regular Legendrian rational tangles which correspond to the vectors $(q_n,q_{n-1}^{f_{n-1}},\ldots,q_1^{f_1})$ and $(q_n^{f_n},q_{n-1}^{f_{n-1}},\ldots,q_1^{f_1})$ are Legendrian isotopic for any $n$, and any choice of $f_n$ with $0 \leq f_n \leq q_n$.
\end{trivialflypes}

In~\cite{T01}, Traynor showed by way of an isotopy sequence that the tangles with vectors $(q_n,q_{n-1},\ldots,q_2,q_1)$ and $(q_n,q_{n-1}^{f_{n-1}},q_{n-2},q_{n-3}^{f_{n-3}},\ldots,q_2^{f_2},q_1)$ are Legendrian isotopic for $n$ odd and $q_j$ even for $j$ odd---a subclass of regular tangles which she refers to as \emph{minimal} tangles.  In fact this same isotopy sequence can be used to extend this result to all regular tangles, omitting the minimality conditions on $n$ and the $q_j$'s.  Theorem~\ref{verticalflypes} is more substantial even than this extension, allowing for a fixed choice of exponents $f_j$, for $j$ odd, common to both of these vectors.  The proof of this theorem is adapted from~\cite{T01}, requiring an additional isotopy sequence given in Figure~\ref{vflypeproof}.

\begin{verticalflypes}
\label{verticalflypes}
If two regular Legendrian rational tangles differ only by vertical Legendrian flypes, then they are Legendrian isotopic.
\end{verticalflypes}

\begin{proof}
We first refer the reader to the isotopy sequence shown in Figure~\ref{vflypeproof}, also adapted from~\cite{T01}.  We will employ two versions of this sequence; one shown explicitly, and the other obtained by taking $b=0$ and rotating the page $180^\circ$ about its vertical axis, yielding the sequence originally given by Traynor.  For brevity, we will refer to these as the primary and secondary sequences, respectively.  

We proceed by induction on the length of the tangles.  The proof divides naturally into two cases, corresponding to the parity of this length.  The result follows for tangles of length two from Proposition~\ref{trivialflypes}, and for tangles of length three from the general form of Traynor's result.  This establishes the base step of each case, as vertical flypes cannot be applied to tangles of length one.  We  prove the induction step for only even-length tangles explicitly, as the proof of the odd-length case is nearly identical, differing only in the particular notation employed.

We assume that the statement is true for all tangles of length $2k$, and seek to prove that 
\[(q_{2k+2},q_{2k+1}^{f_{k+1}},q_{2k},\ldots, q_3^{f_3}, q_2,q_1^{f_1})\simeq (q_{2k+2},q_{2k+1}^{f_{2k+1}},q_{2k}^{f_{2k}},\ldots, q_3^{f_3}, q_2^{f_2},q_1^{f_1}).\]
To establish this equivalence, it suffices to show that  
\[(q_{2k+2},q_{2k+1}^{f_{2k+1}},q_{2k},\ldots, q_3^{f_3}, q_2,0)\simeq(q_{2k+2},q_{2k+1}^{f_{2k+1}},q_{2k}^{f_{2k}},\ldots, q_3^{f_3}, q_2^{f_2},0),\] 
since such an isotopy can then be performed while fixing the portion of the tangle corresponding to $q_1$.  
Furthermore, the induction hypothesis implies that \[(q_{2k+2},q_{2k+1}^{f_{2k+1}},q_{2k},\ldots, q_3^{f_3}, q_2,0)\simeq(q_{2k+2},q_{2k+1}^{f_{2k+1}},q_{2k}^{f_{2k}},\ldots, q_3^{f_3}, q_2,0).\] 
Thus, we need only focus on applying flypes at twist component $q_2$.  To this end, we first show that 
\[(q_{2k+2},q_{2k+1}^{f_{2k+1}},q_{2k}^{f_{2k}},\ldots, q_3^{f_3}, 1,0)\simeq(q_{2k+2},q_{2k+1}^{f_{2k+1}},q_{2k}^{f_{2k}},\ldots, q_3^{f_3}, 1^1,0).\]

If $f_3\neq 0$, then in the primary sequence shown in Figure~\ref{vflypeproof} we set $a=f_3$ and $b=q_3-f_3$, and we take $\G$ be the tangle $(q_{2k+2}, q_{2k+1}^{f_{2k+1}}, q_{2k}^{f_{2k}},\ldots, q_4^{f_4}, 0)$.  We then apply the first three steps of this isotopy sequence, producing the fourth diagram shown.  The proof now divides again, depending on the value of $q_3-f_3$.

\begin{figure}[ht]
\begin{center}
\resizebox{\textwidth}{!}{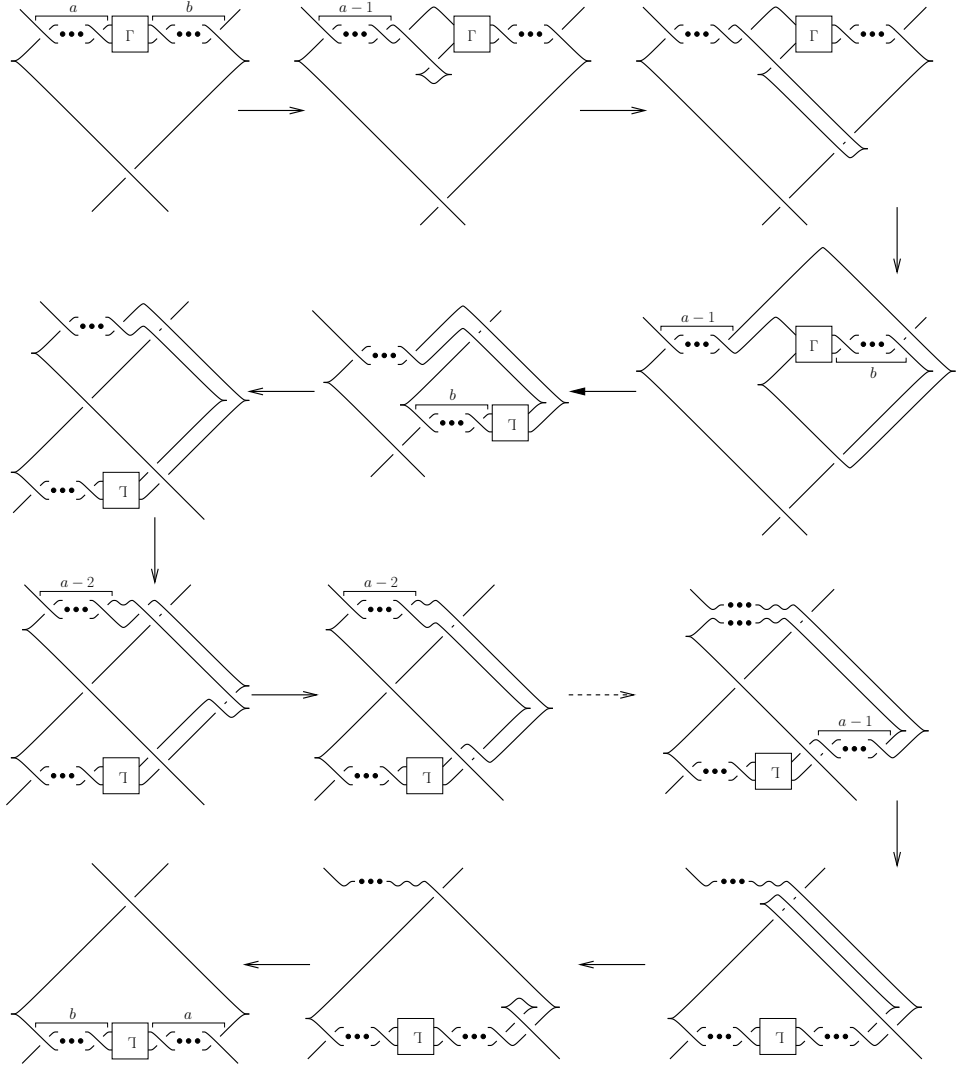}
\caption[The isotopy sequence for vertical Legendrian flypes.]{The primary isotopy sequence used in the proof of Theorem~\ref{verticalflypes}.  The dashed arrow indicates an iteration of the preceding two steps.}
\label{vflypeproof}
\end{center}
\end{figure}

If $q_3-f_3=0$, then there are no horizontal twists adjacent to $\G$ in the figure.  This gives a subtangle consisting of $\G$ with an extra vertical twist below, which can be viewed as the tangle $(q_{2k+2}, q_{2k+1}^{f_{2k+1}}, q_{2k}^{f_{2k}},\ldots, (q_4+1)^{f_4}, 0)$.  An additional application of the induction hypothesis then implies the fourth step of the primary sequence.  

If instead $q_3-f_3\neq 0$, then we still have horizontal twists adjacent to $\G$.  We then set $a=q_3-f_3$ in the secondary sequence described above.  Applying the first three steps of the secondary sequence to the subtangle $(q_{2k+2}, q_{2k+1}^{f_{2k+1}}, q_{2k}^{f_{2k}},\ldots, q_4, q_3-f_3,1,0)$ then produces the desired subtangle $(q_{2k+2}, q_{2k+1}^{f_{2k+1}}, q_{2k}^{f_{2k}},\ldots, (q_4+1)^{f_4}, 0)$, as in the previous case.  The induction hypothesis now provides the fourth step of the secondary sequence, the remainder of which then completes the fourth step of the primary sequence. 

In either case, proceeding through the remaining steps of the primary sequence then produces the desired isotopy.  By alternating between the isotopy described above using the primary and secondary sequences, and a similar isotopy using their rotations by $180^\circ$ about the vertical axis of the page, we have that 
\[(q_{2k+2},q_{2k+1}^{f_{2k+1}},q_{2k}^{f_{2k}},\ldots, q_3^{f_3}, q_2,0)\simeq(q_{2k+2},q_{2k+1}^{f_{2k+1}},q_{2k}^{f_{2k}},\ldots, q_3^{f_3}, q_2^{f_2},0),\] completing the induction step in the even-length case.
\end{proof}

We will refer to the way in which pairs of endpoints of a tangle are connected by the individual strands as the \emph{connectivity type} of a tangle.  Specifically, we say that $\G_q$ is of \emph{connectivity type $0$, $\infty$, or $1$} if $\g_q^1$ connects $p_1$ to $p_2$, $p_3$, or $p_4$, respectively  (see Figure~\ref{connectivitytypes}).  As shown in~\cite{KL02}, given $q=P/Q$, with $P$ and $Q$ relatively prime, $\G_q$ is of connectivity type $0$ if $P$ is even, $\infty$ if $Q$ is even, or $1$ if both are odd.  Since the endpoints of a tangle remain fixed throughout any topological isotopy, the connectivity type of $\G_{q^f}$ will agree with that of $\G_q$ for any $f$.

\begin{figure}[ht]
\begin{center}
\resizebox{0.75\textwidth}{!}{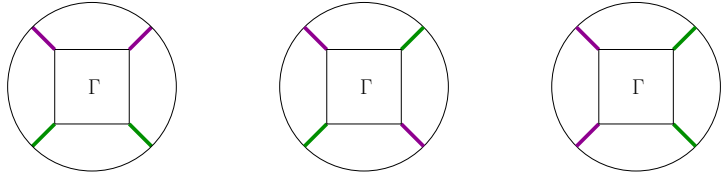}
\caption[Connectivity types of rational tangles.]{Connectivity type 0 (left), 1 (center), and $\infty$ (right).}
\label{connectivitytypes}
\end{center}
\end{figure}

For any tangle $(q_n,q_{n-1}^{f_{n-1}},q_{n-2}^{f_{n-2}},\ldots,q_1^{f_1})$, let $\sigma(f)$ denote the sum of the exponents $f_j$ with $j$ odd, $j<n$.  Further, let $\sigma_{\infty}(f)$ denote the sum of such exponents for which the subtangle $(q_n,q_{n-1},\ldots, q_{j+1},0)$ is of connectivity type $\infty$.

\begin{infconnhorizontalflypes}
\label{infconnhorizontalflypes}
For any $n$, let $\G_{q^f}$ and $\G_{q^g}$ be Legendrian rational tangles of length $n$.  If $\sigma_{\infty}(f)\neq\sigma_{\infty}(g)$, then $\G_{q^f}$ and $\G_{q^g}$ are not Legendrian isotopic.
\end{infconnhorizontalflypes}

\begin{figure}[ht]
\begin{center}
\resizebox{\textwidth}{!}{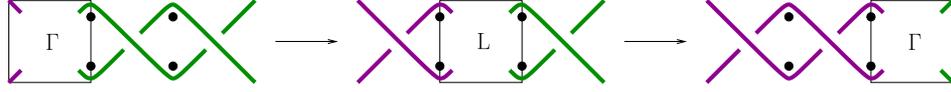}
\caption[Horizontal flypes and connectivity type $\infty$.]{Horizontal flypes change elliptic counts if $\G$ is of connectivity type $\infty$.} \label{infconnflypesproof}
\end{center}
\end{figure}

\begin{proof}
Assume, without loss of generality, that $\sigma_{\infty}(f)<\sigma_{\infty}(g)$.  By construction, the strand $\g_q^1$ enters every remainder rectangle $\r_1$, $\r_2$, \dots, $\r_n$ of $\boxdot_q$ through its top left corner.  Thus, if the subtangle contained in any of these rectangles is of connectivity type $\infty$, the strand $\g_q^2$ must both enter and leave this rectangle through its right side.  Since each flype fixes the endpoints of the preceding subtangle, the iterative constructions of $\G_{q^f}$ and $\G_{q^g}$ from $\G_q$ ensure that each flype performed on any subtangle of connectivity type $\infty$ will produce a tangle with two more elliptic points on its first strand and two fewer on its second, as in Figure~\ref{infconnflypesproof}.   Using the notation of Section~\ref{sec:lemmas}, this implies that $|E_{q^g}^1|-|E_{q^f}^1| = |E_{q^f}^2|-|E_{q^g}^2|=2\big(\sigma_{\infty}(g)-\sigma_{\infty}(f)\big)$.  Since the cardinalities of $E_{q^f}^{\ell}$ and $E_{q^g}^{\ell}$ are distinct, Lemma~\ref{cardinalitylemma} implies that there cannot be a Legendrian isotopy between the given tangles.
\end{proof}

\begin{oddlengthhorizontalflypes}
\label{oddlengthhorizontalflypes}
For $n$ odd, let $\G_{q^f}$ and $\G_{q^g}$ be Legendrian rational tangles of length $n$.  If $\sigma(f)\neq\sigma(g)$, then $\G_{q^f}$ and $\G_{q^g}$ are not Legendrian isotopic.
\end{oddlengthhorizontalflypes}

\begin{proof}
Assume, without loss of generality, that $\sigma(f)<\sigma(g)$.  By construction, the strand $\g_q^1$ enters every remainder rectangle $\r_1$, $\r_2$, \dots, $\r_n$ of $\boxdot_q$ through its top left corner.  A simple induction argument then shows that each successive flype in the iterative construction of $\G_{q^f}$ from $\G_q$ will move the set of $q_n-1$ shared elliptics one point further along the first strand of the tangle, as shown in Figure~\ref{hflypeproof}.  Thus, traversing $\g_q^1$ through $\boxdot_{q^f}$, we will first encounter $\sigma(f)$ tagged dots, followed by the set of $q_n-1$ shared dots.  Similarly, we will encounter $\sigma(g)$ tagged dots along $\g_q^1$ before the $q_n-1$ shared dots in $\boxdot_{q^g}$.

Now, suppose that $\G_{q^f}$ and $\G_{q^g}$ are isotopic.  By Lemma~\ref{bijectionlemma}, we then have a bijection $\beta$ from $E_{q^f}$ to $E_{q^g}$ which restricts to order-preserving bijections $\beta^{\ell}$ from $E_{q^f}^{\ell}$ to $E_{q^g}^{\ell}$, for $\ell\in\{1,2\}$.  In particular, the first shared elliptic in $E_{q^f}^1$ will be mapped to the $(\sigma(f)+1)^{\textrm{th}}$ elliptic point in $E_{q^g}^1$ under $\beta^1$.  However, since $\sigma(f)<\sigma(g)$, this image will not be a shared elliptic in $E_{q^g}^1$, implying that the image of this same shared elliptic under $\beta^2$ must be a distinct point in $E_{q^g}$, contradicting that $\beta$ was well-defined.
\end{proof}

\begin{figure}[ht]
\begin{center}
\resizebox{0.8\textwidth}{!}{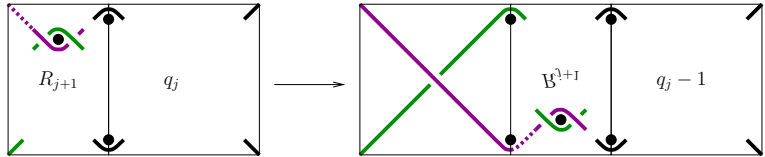}
\caption[The effect of flyping on the shared elliptics.]{Each horizontal flype moves the set of shared elliptics one point further along the strand $\g_q^1$.} 
\label{hflypeproof}
\end{center}
\end{figure} 

\section{Conclusion}
\label{sec:conclusion}
We close with a few remarks regarding the limitations of our techniques and directions for further study.  Combining Theorems~\ref{infconnhorizontalflypes} and~\ref{oddlengthhorizontalflypes} provides an incomplete hierarchy in classifying odd-length tangles.  Specifically, tangles which contain different numbers of horizontal flypes are not isotopic.  Neither are tangles which contain the same number of horizontal flypes in total, but with different numbers of horizontal flypes applied to subtangles of connectivity type $\infty$.  However, if two tangles contain the same number of horizontal flypes, with the same number performed at such subtangles, the classification is presently unknown in any general form.  

To clarify this last remark, we provide a few examples, all illustrated in Figure~\ref{finalexamples}.  The tangles $(2,1,1^1,2,1)$ and $(2,1,1,2,1^1)$ are indistinguishable by our techniques, as the shared elliptics in each lie in the same positions along either strand of the tangles.  It can be shown, however, that these tangles are not isotopic by considering their strandwise classical invariants (cf. Section~\ref{sec:classicalinvariants}).  The same is true of $(3,1,1^1,2,1)$ and $(3,1,1,2,1^1)$, where each horizontal flype is performed at a distinct subtangle of connectivity type $\infty$.  However, the tangles $(2,1,2^1,1,1)$ and $(2,1,2,1,1^1)$ are in fact distinguishable by counting elliptics, as in the proof of Theorem~\ref{oddlengthhorizontalflypes}, though here $\beta^1$ would map the shared elliptic of the first tangle to the shared elliptic of the second.  Instead, $\beta^2$ would map the shared elliptic of the first tangle to a tagged, but not shared, elliptic in the second, providing the desired contradiction.  The strandwise invariants of these tangles are all zero, making our approach useful in this particular example.  The tangles $(2,1,2^2,2,2)$ and $(2,1,2,2,2^2)$ are indistinguishable by counting elliptics, have the same polynomial invariants defined by Traynor, and have zero strandwise classical invariants, showing that a complete classification will yet require more than known techniques.    

It is our suspicion that the converse of Theorem~\ref{verticalflypes} is true for odd-length tangles; that regular Legendrian rational tangles are Legendrian isotopic only if they have the same number of horizontal flypes at each horizontal twist component.  Indeed, this is what we sought to prove with this work, but the above examples illustrate the limitations of our approach to this end.  It is expected that a comprehensive analysis of the effect of a flype in repositioning elliptic singularities along both strands of a tangle will provide a further refinement of this work, allowing us to extend our results to a class of tangles which contains the distinguishable example above, though the combinatorics involved in such an analysis are intimidating.  Furthermore, the classification of even-length tangles beyond Theorems~\ref{verticalflypes} and~\ref{infconnhorizontalflypes} is still an open question.  The absence of shared elliptics in this case is prohibitive, from the perspective of Theorem~\ref{oddlengthhorizontalflypes}.

While the completion of this classification is one avenue of further work, there are two other fundamental questions associated to rational tangles that would be interesting to examine in greater depth.  The first is how our classification corresponds to the classification of Legendrian knots and links obtained as the closures of regular Legendrian rational tangles.  Current work is being done to compare these classifications wherein it is already clear that the notion of Legendrian tangle isotopy is much more restrictive than that of general Legendrian isotopy.  Specifically, there are lots of examples of tangles which are not Legendrian isotopic, but whose closures are.  The other, more difficult, question is how our classification of rational tangles can contribute to a similar classification of algebraic tangles---tangles which are formed by gluing together the endpoints of several rational tangles in some fashion.  It is not immediately clear how the compressing discs we constructed can be adapted to address either of these questions, though the ability to use characteristic foliations as distinguishing invariants, under suitable restrictions, still provides a new perspective on such future research.

\begin{figure}[ht]
\begin{center}
\includegraphics[width=0.85\textwidth]{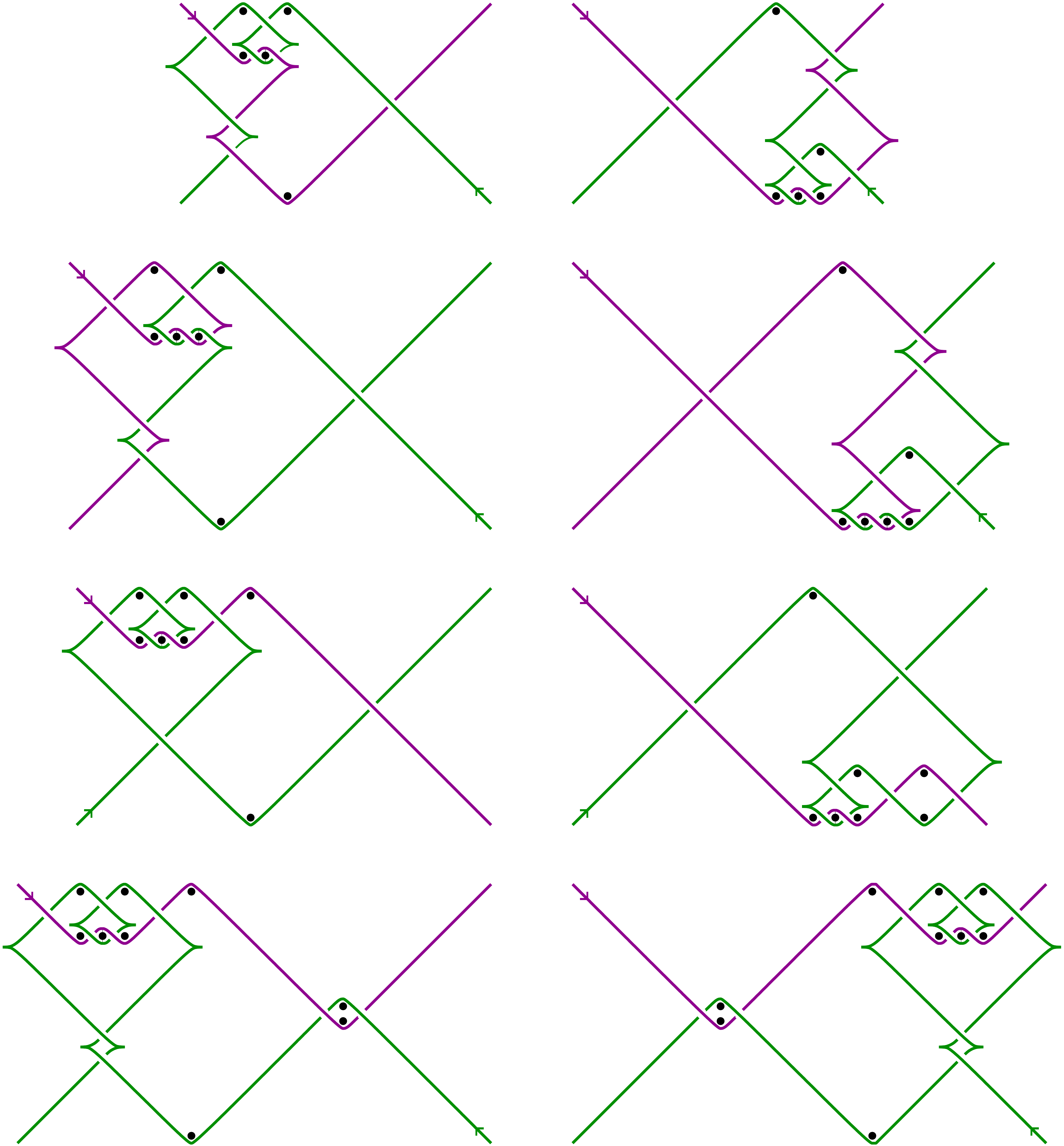}
\caption[Final examples of horizontal flypes.]{The tangles $(2,1,1^1,2,1)$, $(2,1,1,2,1^1)$, $(3,1,1^1,2,1)$, $(3,1,1,2,1^1)$, $(2,1,2^1,1,1)$, $(2,1,2,1,1^1)$, $(2,1,2^1,2,1)$, $(2,1,2,2,1^1)$.}
\label{finalexamples}
\end{center}
\end{figure}

\bibliographystyle{amsplain}
\bibliography{ClassLRTs}

\end{document}